\documentclass[11pt]{article}
\usepackage{amsfonts, amsbsy, amssymb, amsmath, graphicx, amsthm, float, subfigure, natbib, color, url}
\newtheorem{theorem}{Theorem}

\usepackage{geometry}                
\usepackage{graphicx}
\usepackage{amssymb}
\usepackage[english]{babel} 

\usepackage{epstopdf}
\usepackage{ifpdf}
\usepackage{url}

\begin{document}

\bibliographystyle{natbib}

\title{Lagrangian Descriptors for Two Dimensional, Area Preserving, Autonomous and Nonautonomous Maps}
\author{Carlos Lopesino$^{1}$, Francisco Balibrea$^{1}$, Stephen Wiggins$^2$, Ana M. Mancho$^1$\\
$^1$Instituto de Ciencias Matem\'aticas, CSIC-UAM-UC3M-UCM,
\\ C/ Nicol\'as Cabrera 15, Campus Cantoblanco UAM, 28049
Madrid, Spain\\
$^2$School of Mathematics, University of Bristol, \\Bristol BS8 1TW, United Kingdom}
\maketitle
\begin{abstract}
In this paper we generalize the method of Lagrangian descriptors to  two dimensional, area preserving, autonomous and 
nonautonomous discrete time dynamical systems. We consider four generic model problems--a hyperbolic saddle point for a 
linear, area-preserving autonomous map, a hyperbolic saddle point for a nonlinear, area-preserving autonomous map, a 
hyperbolic saddle point for linear, area-preserving nonautonomous map, and a hyperbolic saddle point for nonlinear, 
area-preserving nonautonomous map. The discrete time setting allows us to evaluate the expression for the Lagrangian 
descriptors explicitly for a certain class of norms.  This enables us to  provide a rigorous setting for the notion that 
the `singular sets'' of the Lagrangian descriptors  correspond to the stable and unstable manifolds of hyperbolic 
invariant sets, as well as to understand how  this depends upon the particular norms that are used. Finally we analyze, 
from the computational point of view, the performance of this tool for general nonlinear maps, by computing the 
``chaotic saddle''  for autonomous and nonautonomous versions of  the H\'enon map.
\end{abstract}

\section{Introduction}\label{sec:intro}

Lagrangian descriptors (also referred to in the literature as the ''M function'') were first introduced as a tool for 
finding hyperbolic trajectories in \cite{chaos}.  In this paper the notion of  {\em distinguished trajectory} was  
introduced as a generalization of the well-known idea of distinguished {\em hyperbolic} trajectory. The numerical 
computation of distinguished trajectories was discussed in some detail, and applications to known benchmark examples, as 
well as to geophysical fluid flows defined as data sets  were also given. 
Later \cite{prl} showed that it could be used to reveal Lagrangian invariant structures in realistic fluid flows. In 
particular, a geophysical data set in the region of the Kuroshio current was analysed and it was shown that Lagrangian 
descriptors could be used to reveal the “Lagrangian skeleton” of the flow, i.e. hyperbolic and elliptic regions, as well 
as the invariant manifolds that delineate these regions.  A deeper study of the Lagrangian transport issue associated 
with the Kuroshio using Lagrangian descriptors is given in \cite{jfm}. Advantages of the method over finite time 
Lyapunov exponents (FTLE) and finite size Lyapunov exponents (FSLE) were also discussed.

Since then Lagrangian descriptors have been further developed and their ability to reveal phase space structures in 
dynamical systems more generally has been confirmed. In particular, Lagrangian descriptors are used in \cite{amism11} to 
 reveal the Lagrangian structures that define transport routes across the Antarctic polar vortex. Further studies of  
transport issues related to the Antarctic polar vortex using Lagrangian descriptors are given in \cite{ammsi13} where 
vortex Rossby wave breaking is related to Lagrangian structures. In \cite{rempel} Lagrangian descriptors are used to 
study  the influence of coherent structures on the saturation of a nonlinear dynamo. In \cite{mmw14} Lagrangian 
descriptors are used to analyse the influence of Lagrangian structure on the transport of buoys in the Gulf stream and 
in a region of the Gulf of Mexico relevant to the Deepwater Horizon oil spill. In \cite{mwcm13} a detailed analysis of 
the behaviour of Lagrangian descriptors is provided in terms of benchmark problems, new Lagrangian descriptors are 
introduced,  extension of Lagrangian descriptors to 3D flows is given (using the time dependent Hill’s spherical vortex 
as a benchmark problem), and a detailed analysis and discussion of the computational performance (with a comparison with 
FTLE) is presented.

Lagrangian descriptors are based on the integration, for a finite time, along trajectories of an intrinsic bounded, 
positive geometrical and/or physical property of the trajectory itself, such as the norm of the velocity, acceleration, 
or curvature. Hyperbolic structures are revealed as singular features of the contours of the Lagrangian descriptors, but 
the sharpness of these singular features depends on the particular norm chosen. These issues were explored in 
\cite{mwcm13}, and further explored in this paper.

All of the work thus far on Lagrangian descriptors has been in the continuous time setting. In this paper we generalize 
the method of Lagrangian descriptors to the discrete time setting of two dimensional area preserving maps, both 
autonomous and nonautonomous, and provide  theoretical support for their perfomance.

This paper is organized as follows. In section \ref{sec:DLDdef} we defined discrete Lagrangian descriptors. We then 
consider four examples. In  section \ref{sec:examp1} we consider a linear autonomous area preserving map have a 
hyperbolic saddle point at the origin, in  \ref{sec:examp2} we consider a nonlinear autonomous area preserving map have 
a hyperbolic saddle point at the origin, in  \ref{sec:examp3} we consider a linear nonautonomous area preserving map 
have a hyperbolic saddle  trajectory  at the origin, and in \ref{sec:examp4} we consider a nonlinear nonautonomous area 
preserving map have a hyperbolic trajectory at the origin. For each example we show that the Lagrangian descriptors 
reveal the stable and unstable manifolds by being singular on the manifolds. The notion of  ``being singular'' is made 
precise in Theorem \ref{thm:lin_aut_map}. In section \ref{sec:henon}  we explore further the method beyond the 
analytical examples. We use  discrete Lagrangian  descriptors to computationally reveal the chaotic saddle of the 
H\'enon map, and in section \ref{sec:NAhenon} we consider a nonautonomous version of the H\'enon map.  In section 
\ref{sec:summ} we summarize the conclusions and suggest future directions for this work.

\section{Lagrangian Descriptors for Maps}
\label{sec:DLDdef}

Let

\begin{equation}
\{ x_n, y_n \}_{n=-N}^{n=N}, \quad N \in \mathbb{N},
\label{orbit}
\end{equation}

\noindent
denote an orbit of length $2N +1$ generated by a two dimensional map. At this point it does not matter whether or not 
the map is autonomous or nonautonomous. The method of Lagrangian descriptors applies to orbits in general, regardless of 
the type of dynamics that generate the orbit.

The first Lagrangian descriptor  (also known as the ``$M$ function'') for continuous time systems was based on computing
the arclength of trajectories for a finite time (\cite{chaos}). Extending this idea to maps is straightforward, and the
corresponding discrete Lagrangian descriptor (DLD) is given by:

\begin{equation}
MD_2 = \sum^{N-1}_{i=-N} \sqrt{ (x_{i+1}-x_i)^2 + (y_{i+1}-y_i)^2 }.
\label{eq:DLD_al}
\end{equation}

\noindent
In analogy with the  work on continuous time Lagrangian descriptors in \cite{mwcm13}, we consider different norms for 
the discretized arclength as follows:

\begin{equation}
MD_p = \sum^{N-1}_{i=-N} \sqrt[p]{ |x_{i+1}-x_i|^p + |y_{i+1}-y_i|^p }, \quad p>1,
\label{eq:DLD_p>1}
\end{equation}

\noindent
and

\begin{equation}
MD_p = \sum^{N-1}_{i=-N} |x_{i+1}-x_i|^p + |y_{i+1}-y_i|^p , \quad p \leq 1.
\label{eq:DLD_p<1}
\end{equation}

\noindent
Considering the space of orbits as a sequence space, \eqref{eq:DLD_p>1} and \eqref{eq:DLD_p<1} are the $\ell^p$ norms
of an orbit.

Henceforth,  we will consider only the case $p \leq 1$ since the proofs are more simple in this case. Now we will 
explore these definitions in the context of  some easily  understood, but generic, examples.

\subsection{Example 1: A Hyperbolic Saddle Point for Linear, Area-Preserving Autonomous Maps}
\label{sec:examp1}

\subsubsection{Linear Saddle point}

Consider the following linear, area-preserving autonomous map:

\begin{equation}
\left \{ \begin{array}{ccc}
   x_{n+1} & = & \lambda x_n,\\
   y_{n+1} & = & \frac{1}{\lambda} y_n, \\
\end{array}\right .
\label{eq:lin_aut_map}
\end{equation}

\noindent
where we will take  $\lambda > 1$. Note that this map is area-preserving, but  area-preservation was not used in the
definition of the DLD's above.

Now we will compute \eqref{eq:DLD_p<1} for this example. Towards this end, we introduce the notation

$$MD_p = MD^{+}_p + MD^{-}_p$$

\noindent
where

$$MD^{+}_p = \sum ^{N-1}_{i=0}  |x_{i+1}-x_i|^p + |y_{i+1}-y_i|^p,$$

\noindent
and

$$MD^{-}_p = \sum ^{-N}_{i=-1}  |x_{i+1}-x_i|^p + |y_{i+1}-y_i|^p.$$

\noindent
We begin by computing $MD^{+}_p$. The computation of  $MD^{-}_p$ is completely analogous, and  therefore we will not
provide the details. We have:

\begin{eqnarray}
  MD^{+}_p & = & \displaystyle{\sum ^{N-1}_{i=0}  |x_{i+1}-x_i|^p + |y_{i+1}-y_i|^p} \nonumber\\
	        &   & 						     \nonumber\\
            & = & \displaystyle{|x_1-x_0|^p + |y_1-y_0|^p + ... + |x_{N}-x_{N-1}|^p + |y_{N}-y_{N-1}|^p}\nonumber\\
            &   &						   \nonumber  \\
            & = & \displaystyle{|\lambda x_0-x_0|^p + |1/\lambda y_0-y_0|^p + ... + |\lambda^{N} x_0 - \lambda
^{N-1}x_0|^p + |1/\lambda^{N} y_0 - 1/\lambda^{N-1} y_0|^p}\nonumber\\
            &   &						     \nonumber\\
            & = & \displaystyle{|x_0|^p|\lambda-1|^p \left (1+\lambda^p+...+\lambda^{(N-1)p}\right ) +
|y_0|^p|1/\lambda-1|^p \left (1+1/\lambda^p+...+1/\lambda^{(N-1)p}\right )}\nonumber\\
	        &   &						    \nonumber \\
	        & = & \displaystyle{|x_0|^p|\lambda-1|^p \left (\frac{\lambda^{Np}-1}{\lambda^p-1}\right ) +
|y_0|^p|1/\lambda-1|^p \left (\frac{1/\lambda^{Np}-1}{1/\lambda^p-1}\right )}\nonumber
\end{eqnarray}
\normalsize

\noindent
where in the last step we have used that the sums are geometric with rates $\lambda^p$ and $1/\lambda^p$, respectively.
By completely analogous calculations we obtain $MD^{-}_p$ as:

$$MD^{-}_p = |x_0|^p|1/\lambda-1|^p \left (\frac{1/\lambda^{Np}-1}{1/\lambda^p-1}\right ) +
|y_0|^p|\lambda-1|^p \left (\frac{\lambda^{Np}-1}{\lambda^p-1}\right ).$$

\noindent
Putting the two terms together, we obtain:

\begin{equation}
\begin{array}{ccl}
       MD_p & = & MD^{+}_p + MD^{-}_p 				     \\
	        &   & 						     \\
            & = & (|x_0|^p+|y_0|^p)(|\lambda-1|^p \left (\frac{\lambda^{Np}-1}{\lambda^p-1}\right ) +
|1/\lambda-1|^p\left (\frac{1/\lambda^{Np}-1}{1/\lambda^p-1}\right ))\\
	        &   & 						     \\
            & = & (|x_0|^p+|y_0|^p)f(\lambda,p,N),\\
\end{array}
\label{DLD_lin_aut}
\end{equation}

\noindent
where $\lambda$, $p$ and $N$ are fixed.

Extensive numerical simulations in a variety of examples (cf. \cite{chaos, prl, nlpg2, amism11, jfm, mwcm13, mmw14})
have shown that ``singular features'' of Lagrangian descriptors correspond to stable and unstable manifolds of
hyperbolic trajectories. We can make this statement rigorous and precise in the context of this example.

\begin{theorem}
Consider a vertical line perpendicular to the unstable manifold of the origin. In particular, consider an arbitrary
point $x = \bar{x}$ and a line parallel to the $y$ axis passing through this point. Then the derivative of $MD_p$, $p
<1$, along this line becomes unbounded on the unstable manifold of the origin.

Similarly, consider a horizontal line perpendicular to the stable manifold of the origin. In particular, consider an
arbitrary point $y = \bar{y}$ and a line parallel to the $x$ axis passing through this point. Then the derivative of
$MD_p$, $p <1$, along this line becomes unbounded on the stable manifold of the origin.
\label{thm:lin_aut_map}
\end{theorem}

\begin{proof} This is a simple calculation using \eqref{DLD_lin_aut} and the fact that $p<1$. This is illustrated in
Figure \ref{Saddle_vs_MDp}.
\end{proof}

\begin{figure}[H]
\centering
{\includegraphics[scale = 0.4]{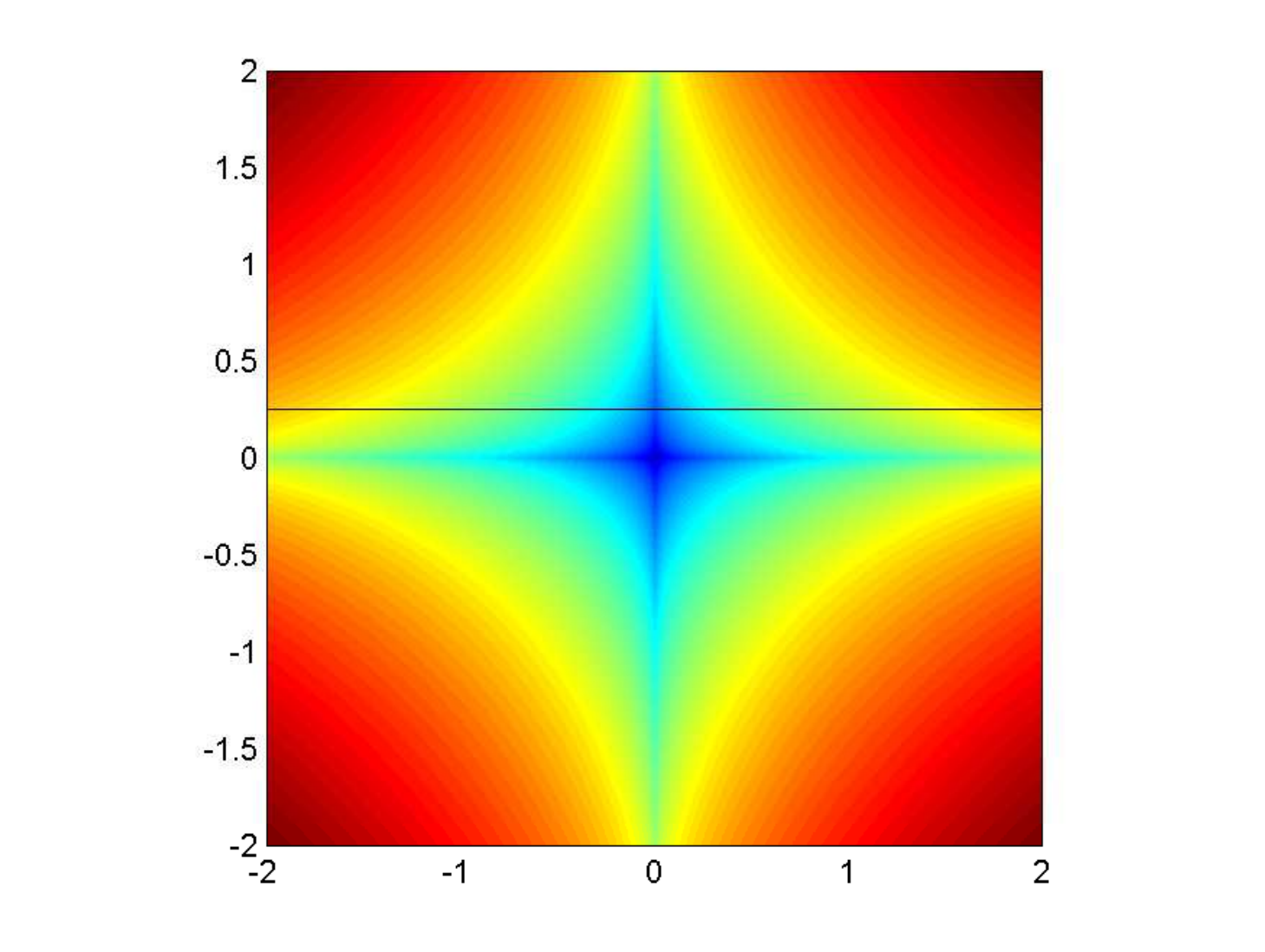}}
{\includegraphics[scale = 0.4]{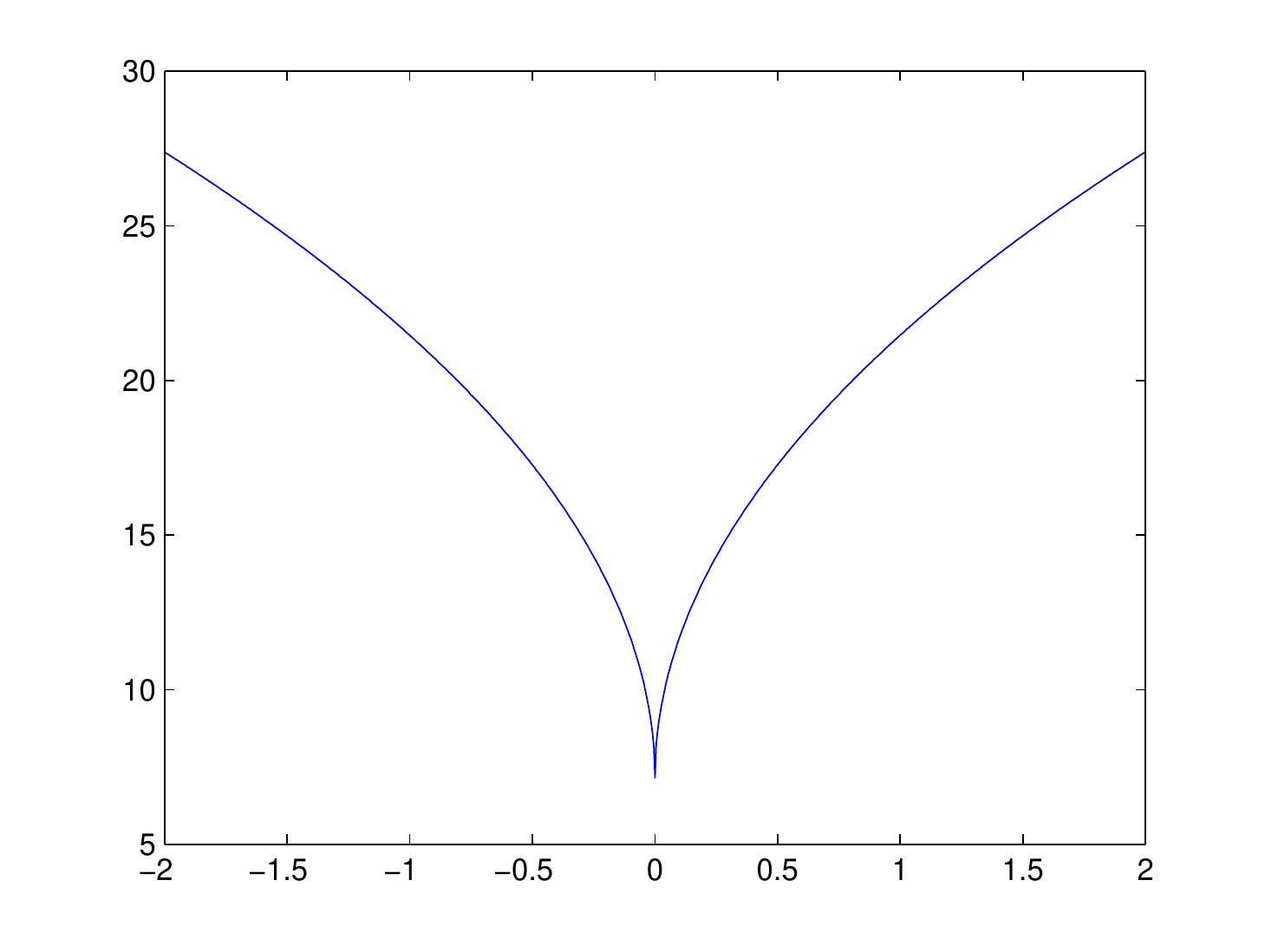}}
\caption{The left-hand panel shows contours of $MD_p$ for $p=0.5$, $N=20$ and $\lambda=1.1$, with a grid point spacing 
of $0.005$. The horizontal black line is at $y=0.25$. The right-hand panel shows the graph of $MD_p$ along this 
horizontal black line, which illustrates the singular nature of the derivative of $MD_p$ on the stable manifold across 
the line $x=0$.}
\label{Saddle_vs_MDp}
\end{figure}

\subsubsection{Linear Rotated Saddle point}

In the example studied  in the previous section the DLD is singular along the stable and unstable manifolds for any 
iteration $n$. However,  the results discussed in \cite{ prl,mwcm13} for the continuous time case show that the 
manifolds are observed for $\tau$  ``sufficiently large'', which is related to a large number of iterations in the 
discrete time case. We explore further these connections by  studying  the case of the rotated saddle point. In order to 
establish a direct link to the continuous time case, we consider  the limits of  small and large numbers  of iterations, 
and  $\lambda \approx 1$.

\noindent
We have the following discrete dynamical system:

  \begin{equation}\label{mapF}
  F(x,y) = A
           \left(
             \begin{array}{c}
             x \\
               \\
             y \\
             \end{array}
           \right)
  \end{equation}

  \noindent
  where

  \begin{equation}\label{matrizA}
    A = \left(
          \begin{array}{cc}
            \frac{1}{\lambda} + \lambda  & \frac{1}{\lambda} - \lambda \\
                                         &                             \\
            \frac{1}{\lambda} - \lambda  & \frac{1}{\lambda} + \lambda \\
          \end{array}
        \right)= \frac{1}{2\lambda} \left(
                  \begin{array}{cc}
                    1+\lambda^2 & 1-\lambda^2 \\
                                &             \\
                    1-\lambda^2 & 1+\lambda^2 \\
                  \end{array}
                \right)
  \end{equation}

  \vspace{3mm}
  \noindent
  in our case with $\lambda > 1$. It is easy to see that the stable and the unstable manifolds are given by the vectors 
$(1,1)$ and $(1,-1)$ respectively. We want to compute $A^{i}-A^{i-i}$ in order to get the expressions of the DLD:
  \begin{equation}\label{DLD}
  MD_p = \sum ^{N-1}_{i=-N}  |x_{i+1}-x_i|^p + |y_{i+1}-y_i|^p
  \end{equation}
  and to find where the 'singularities' are produced and why.

  We know that A can be diagonalized so there exist $D$ and $T$ such that
  \begin{equation}\label{diagonalization}
    D = T^{-1} \cdot A \cdot T
  \end{equation}
  where $D$ is a diagonal matrix. Therefore we got the next expression
  $$D^i = T^{-1} \cdot A^i \cdot T, \quad \text{for every } i.$$
  which is equivalent to
  \begin{equation}\label{A_i}
    A^i = T \cdot D^i \cdot T^{-1}, \quad \text{for every } i.
  \end{equation}
  It is clear that the matrix $T$ is
  \begin{equation}\label{matrizT}
    T = \left(
          \begin{array}{cc}
            1  & 1 \\
            -1 & 1 \\
          \end{array}
        \right)
  \end{equation}
  and therefore
  $$T^{-1} = \frac{1}{2}\left (
                        \begin{array}{cc}
                          1 & -1 \\
                          1 &  1 \\
                        \end{array}
                        \right )$$
  We can check equation \eqref{diagonalization}
  \begin{equation}\label{matrizD}
  \small{
    D = \frac{1}{4\lambda}\left(
                   \begin{array}{ccc}
                     1 & & -1 \\
                       & &    \\
                     1 & &  1 \\
          \end{array}
        \right) \left(
                \begin{array}{cc}
                  1+\lambda^2 & 1-\lambda^2 \\
                              &             \\
                  1-\lambda^2 & 1+\lambda^2 \\
                \end{array}
                \right) \left(
                        \begin{array}{ccc}
                          1  & & 1 \\
                             & &   \\
                          -1 & & 1 \\
                        \end{array}
                        \right) = \left(
                                  \begin{array}{cc}
                                    \lambda & 0                 \\
                                            &                   \\
                                    0       & \frac{1}{\lambda} \\
                                  \end{array}
                                  \right)
  }
  \end{equation}
  So we can guess now how is $A^i$ using equation \eqref{A_i}
  \begin{equation}
  \footnotesize{
    A^i = \frac{1}{2}\left(
                     \begin{array}{ccc}
                       1  & & 1 \\
                          & &   \\
                       -1 & & 1 \\
                     \end{array}
                     \right) \left(
                     \begin{array}{cc}
                       \lambda^i & 0                   \\
                                 &                     \\
                       0         & \frac{1}{\lambda^i} \\
                     \end{array}
                     \right) \left(
                             \begin{array}{ccc}
                               1 & & -1 \\
                                 & &    \\
                               1 & &  1 \\
                             \end{array}
                             \right) = \frac{1}{\lambda^i}\left(
                                                          \begin{array}{cc}
                                                            1+\lambda^{2i} & 1-\lambda^{2i} \\
                                                                           &                \\
                                                            1-\lambda^{2i} & 1+\lambda^{2i} \\
                                                          \end{array}
                                                          \right)
  }
  \end{equation}
  Therefore
  \begin{equation}\label{A_i_expression}
    A^{i}-A^{i-1} = \frac{1}{\lambda^i}\left(
                                     \begin{array}{cc}
                                       \lambda^{2i}-\lambda^{2i-1}-\lambda+1  & -\lambda^{2i}+\lambda^{2i-1}-\lambda+1 \\
                                                    &                         \\
                                       -\lambda^{2i}+\lambda^{2i-1}-\lambda+1 &
                                       \lambda^{2i}-\lambda^{2i-1}-\lambda+1  \\
                                     \end{array}
                                     \right)
  \end{equation}

Now we are going to study the analytical expression of the stable and unstable manifold. For that purpose we will 
develop only $MD^{+}_p$ expression ($MD^-_p$ is analogous). So we have to keep in mind the expression for $MD^{+}_p$ 
that is

  \begin{equation}
  MD^{+}_p = \sum ^{N-1}_{i=0}  |x_{i+1}-x_i|^p + |y_{i+1}-y_i|^p
  \end{equation}
  therefore using equation \eqref{A_i_expression} for $N \geq 1$

  \begin{equation}\label{suma_A_n}
  \scriptsize{
  \begin{array}{rl}
    MD^{+}_p = & \displaystyle{\sum^{N-1}_{i=0}} \frac{1}{\lambda^{(i+1)p}}|(\lambda^{2(i+1)}-\lambda^{2(i+1)-1}-\lambda+1)x_0 + (-\lambda^{2(i+1)}+\lambda^{2(i+1)-1}-\lambda+1)y_0|^p \\
               &   \\
               & + \frac{1}{\lambda^{(i+1)p}}|(-\lambda^{2(i+1)}+\lambda^{2(i+1)-1}-\lambda+1)x_0 +
    \lambda^{2(i+1)}-\lambda^{2(i+1)-1}-\lambda+1)y_0|^p \\
  \end{array}}
  \end{equation}

  \vspace{3mm}
  \noindent
Each term on this sum has singularities  along two different lines. In particular, for each $i$ and $\lambda$, we have 
the two singular lines

  \begin{equation}\label{slope_m}
    y_0 = \frac{\lambda^{2(i+1)}-\lambda^{2(i+1)-1}-\lambda+1}{\lambda^{2(i+1)}-\lambda^{2(i+1)-1}+\lambda-1}x_0 = m(\lambda,i)x_0
  \end{equation}

  \noindent
  and

  \begin{equation}\label{slope_-m}
    y_0 = \frac{1}{m(\lambda,n)}x_0
  \end{equation}

  \noindent
  where $m(\lambda,i)$ and $\displaystyle{\frac{1}{m(\lambda,i)}}$ are, respectively, the slopes of the singular lines.
  If we fix $ \lambda = \lambda_0$ and we increase the number of iterations, we can see the evolution of the singular 
features to the limit shown in Figure \ref{sequence_of_DLD}

  \begin{equation}\label{slope_1}
    \lim_{i \to \infty}m(\lambda_0,i) = 1
  \end{equation}

  \noindent
This convergence is reached rapidly and, for example, for $\lambda=1.1$ it is noticeable from $i=20$ onwards. Thus at 
large $i$ most of the terms in the summation \eqref{suma_A_n} contribute with the same slope, i.e., \eqref{slope_1}, 
Therefore the  contributions of terms in the summation \eqref{suma_A_n} with small $i$  are small and make little impact 
in the global sum \eqref{suma_A_n}. If $i$ is small, the number of terms  contributing  to the DLD is small, and each 
term is a $C^0$ function with discontinuities along  {\em different} lines. Since all terms contribute the same to the 
total pattern, no particular feature is highlighted (see Figure \ref{sequence_of_DLD}b) and \ref{sequence_of_DLD}c)).

The limit $\lambda \approx 1$ is closely related to the Lagrangian Descriptors defined for the continuous time case. 
This  can be seen by considering the limit and noting that $\lambda$ quantifies the separation of points as they are 
iterated and relating this to the arclength integral for the linear saddle point discussed in \cite{mwcm13}.

For any  $i = n_0$  fixed,  it is possible to find a  $\lambda$ in the limit close  to 1 that makes the slope $m$ close 
to the limit value:

  \begin{equation}
    \lim_{\lambda \to 1}m(\lambda,n_0) = 0
  \end{equation}

  \noindent
  In this case, equations \eqref{slope_m} and \eqref{slope_-m} tend to $y=0$ and $x=0$, respectively. The approach to 
this limit  can be observed in the sequence of images shown Figure \ref{sequence_of_DLD} and the DLD derivative along 
the line $y=0.25$ shown in  Figure \ref{derivative_of_DLD}.

\begin{figure*}[htbp!]
  \centering
  \subfigure[DLD $\lambda=1.1$ and $i=1$]{\includegraphics[width=0.3\linewidth]{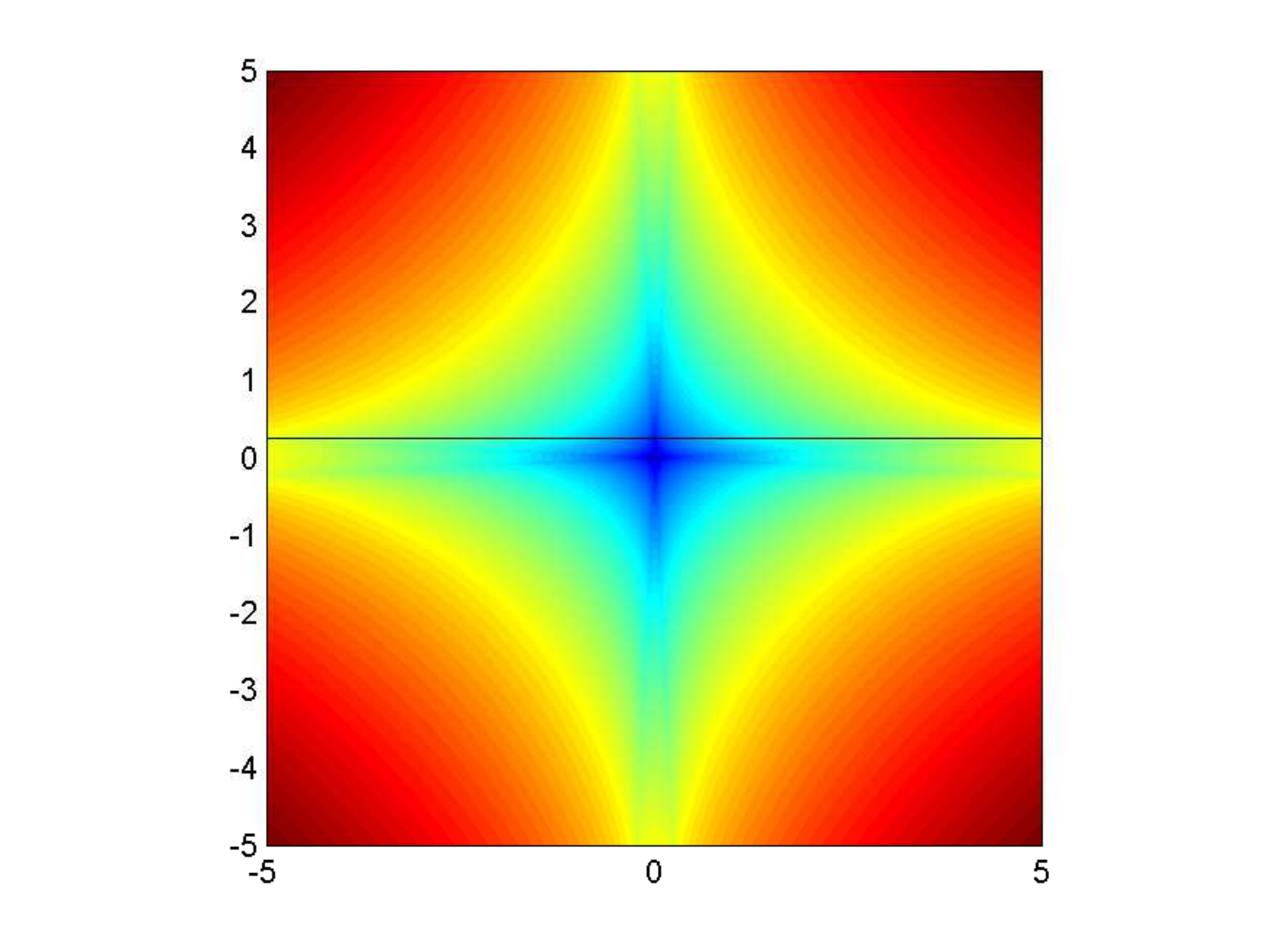}}
  \subfigure[DLD $\lambda=1.1$ and $i=5$]{\includegraphics[width=0.3\linewidth]{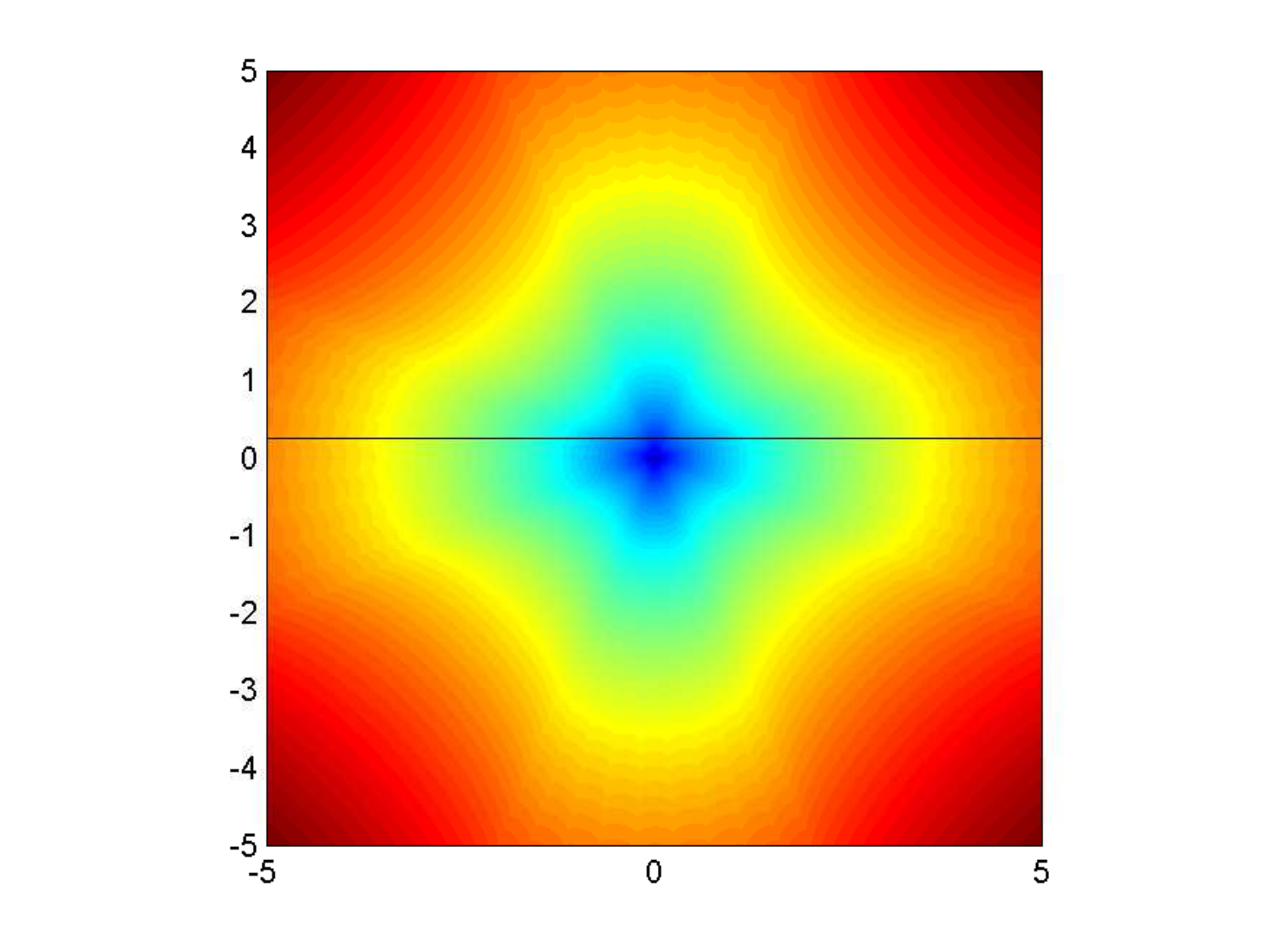}} \\
  \subfigure[DLD $\lambda=1.1$ and $i=10$]{\includegraphics[width=0.3\linewidth]{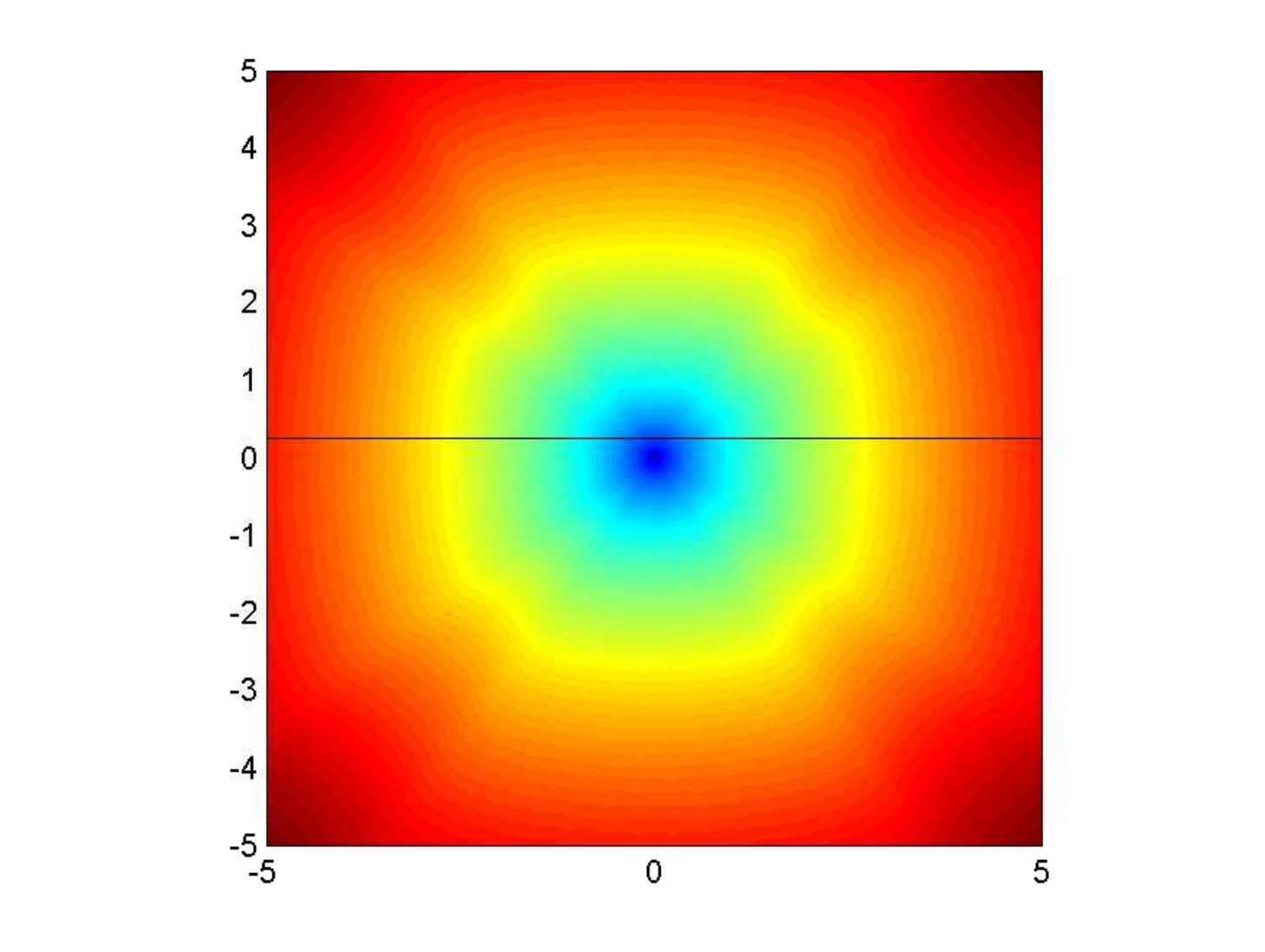}}
  \subfigure[DLD $\lambda=1.1$ and $i=20$]{\includegraphics[width=0.3\linewidth]{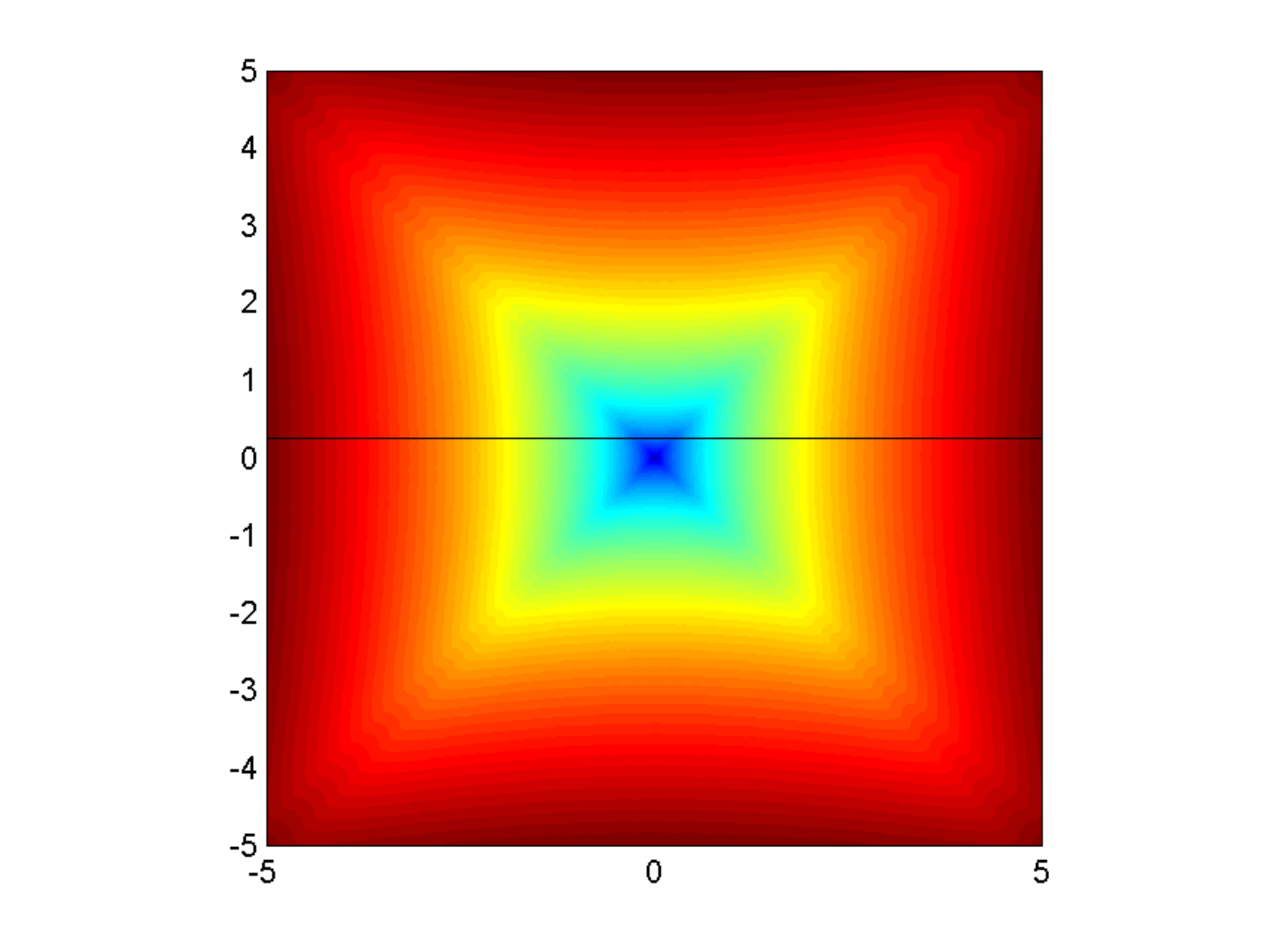}} \\
  \subfigure[DLD $\lambda=1.1$ and $i=30$]{\includegraphics[width=0.3\linewidth]{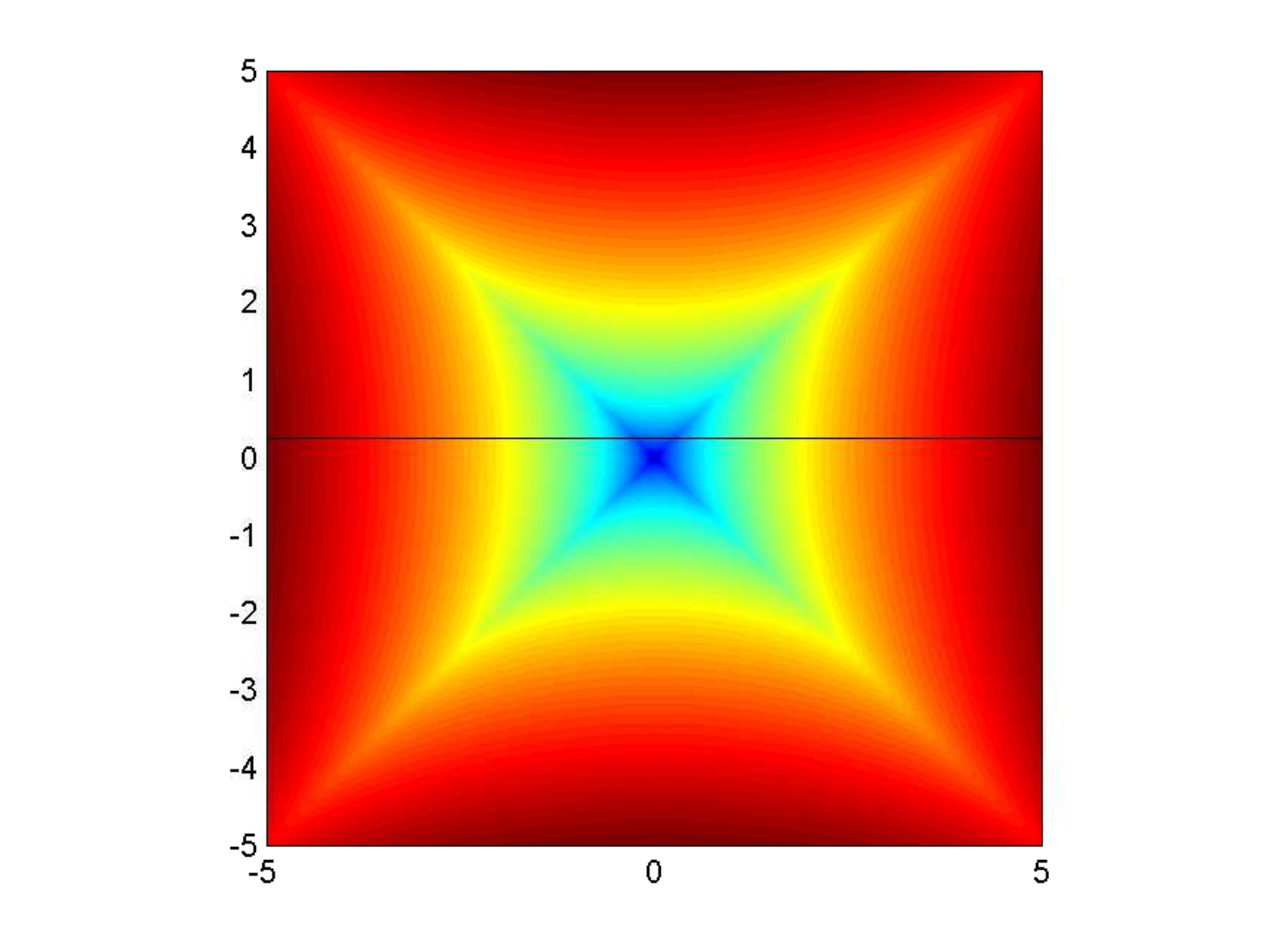}}
  \subfigure[DLD $\lambda=1.1$ and $i=100$]{\includegraphics[width=0.3\linewidth]{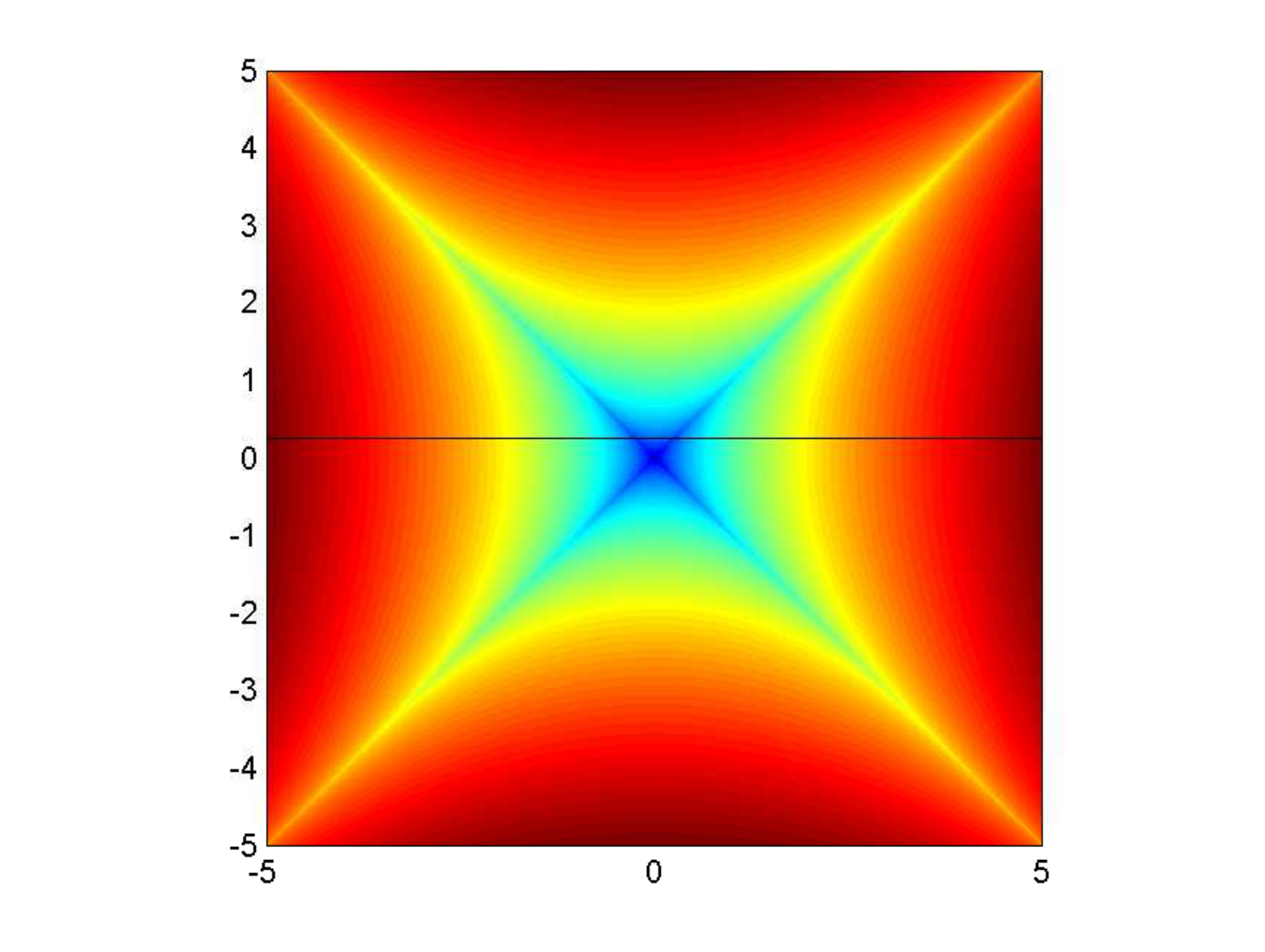}}\\
  \caption{DLD for different values of $\lambda$ and iterations $i$.}
  \label{sequence_of_DLD}
\end{figure*}

\begin{figure*}[htbp!]
  \centering
  \subfigure[Derivative of the DLD  for $\lambda=1.1$ and $i=1$]{\includegraphics[width=0.3\linewidth]{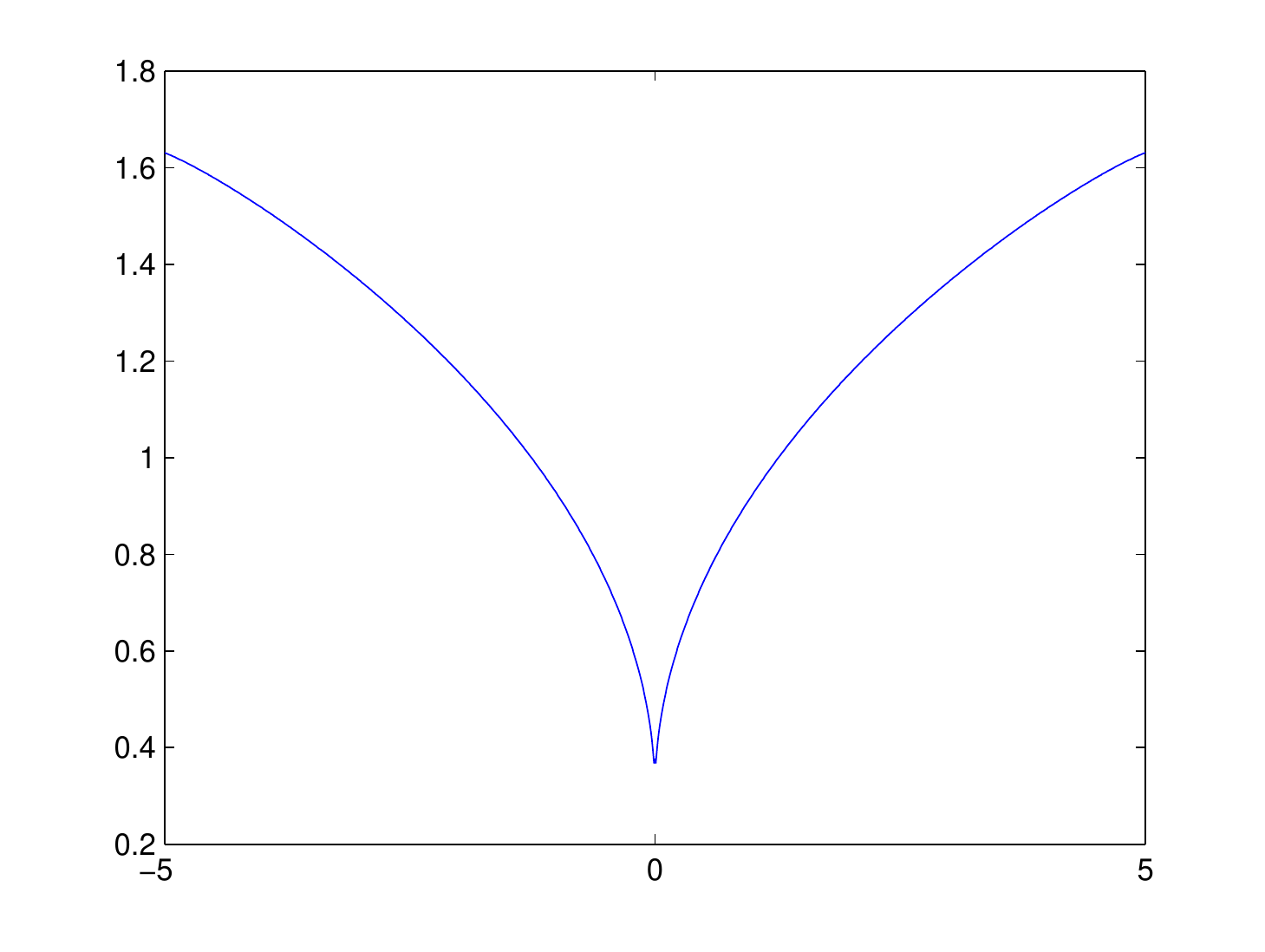}}
  \subfigure[Derivative of the DLD  for $\lambda=1.1$ and $i=5$]{\includegraphics[width=0.3\linewidth]{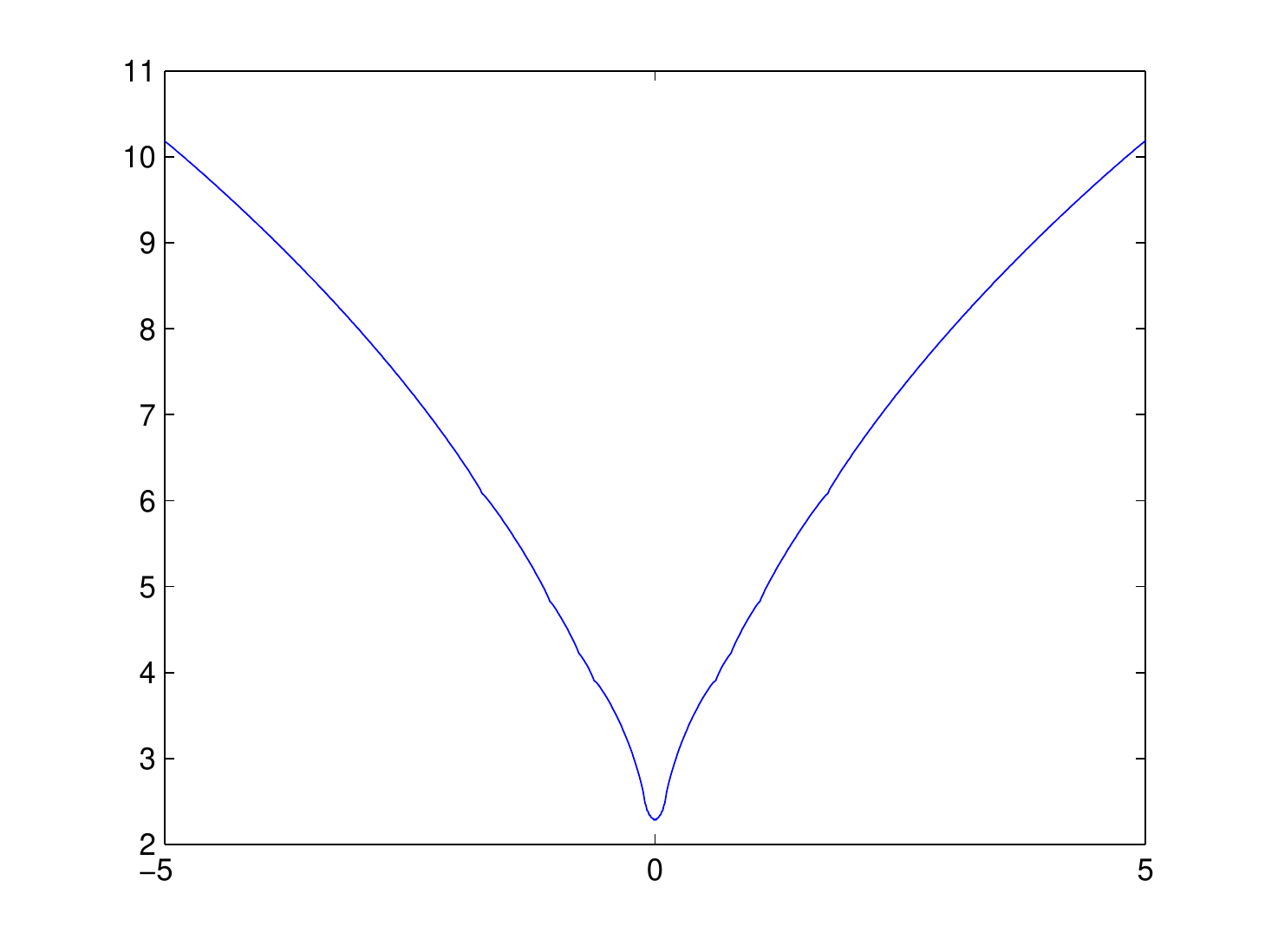}}\\
  \subfigure[Derivative of the DLD  for$\lambda=1.1$ and $i=10$]{\includegraphics[width=0.3\linewidth]{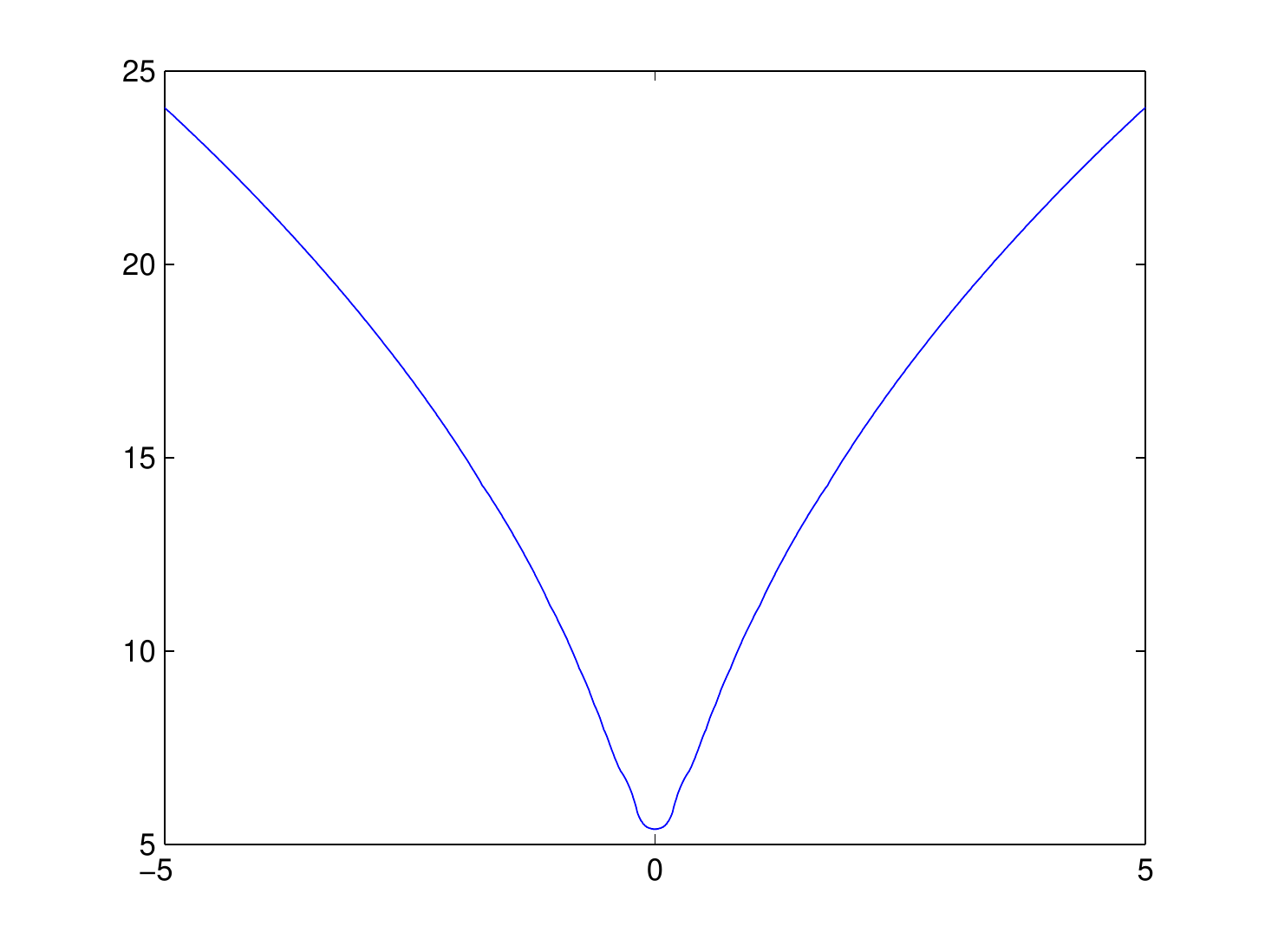}}
  \subfigure[Derivative of the DLD  for $\lambda=1.1$ and $i=20$]{\includegraphics[width=0.3\linewidth]{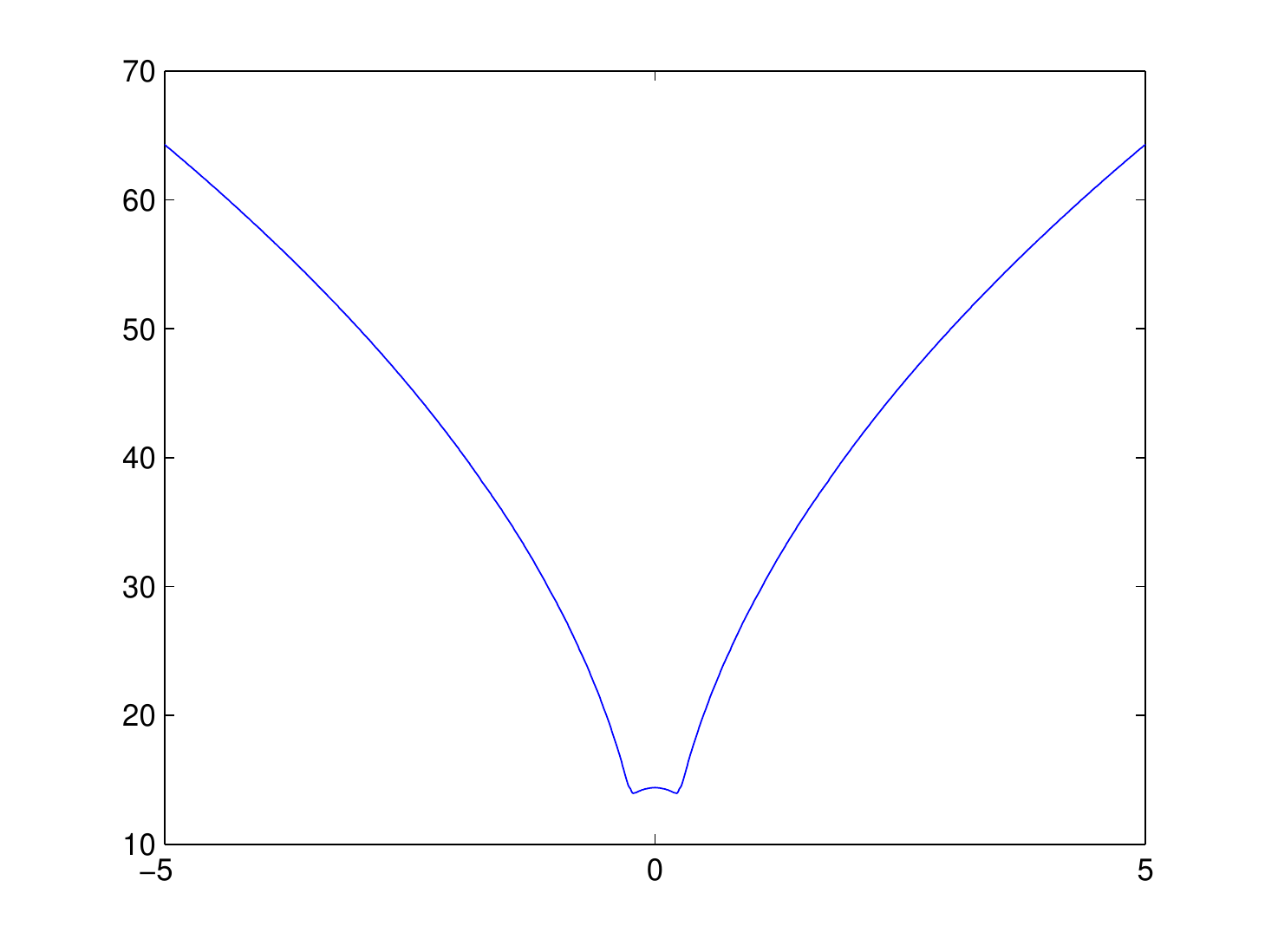}}\\
  \subfigure[Derivative of the DLD  for $\lambda=1.1$ and $i=30$]{\includegraphics[width=0.3\linewidth]{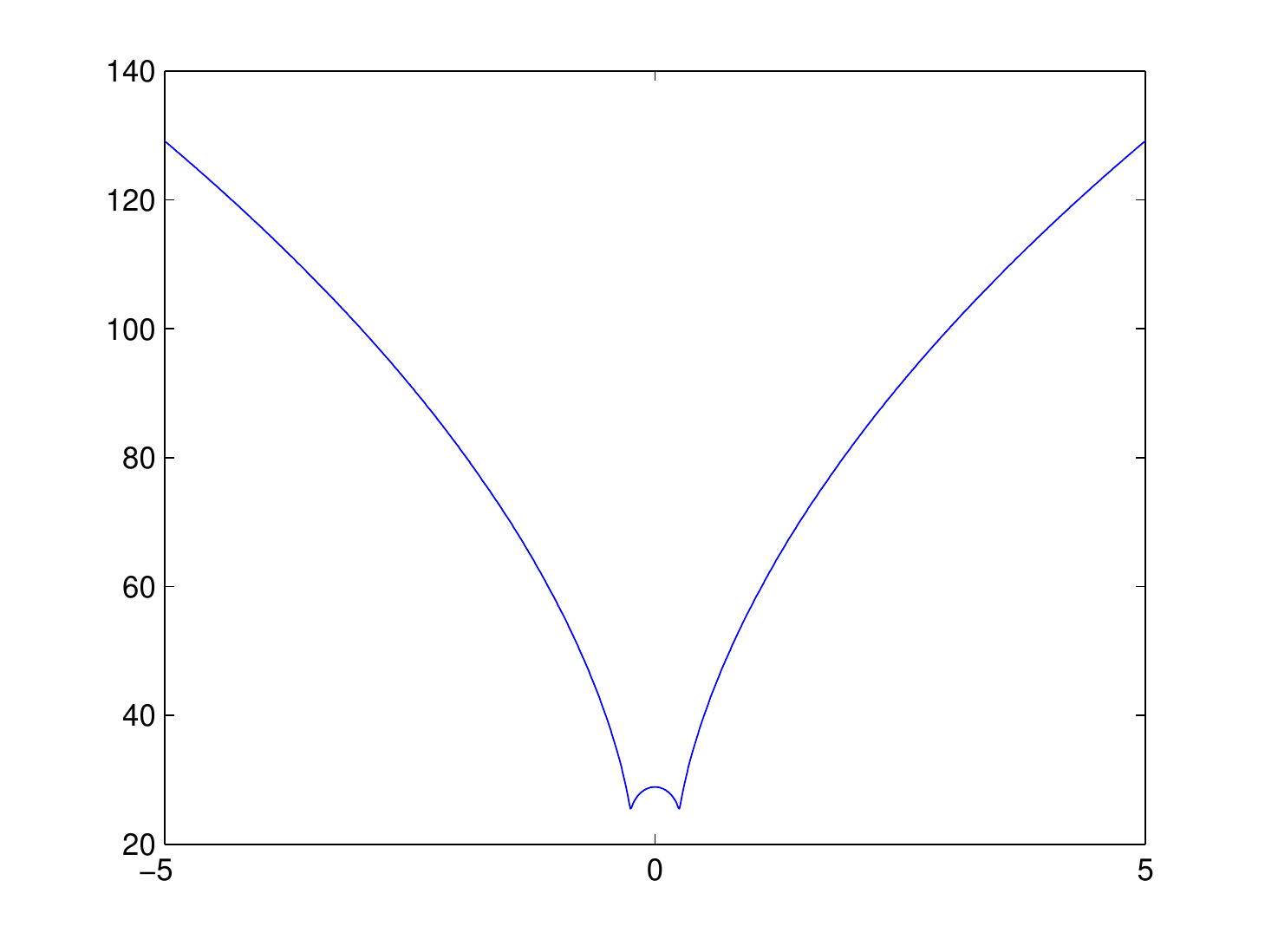}}
  \subfigure[Derivative of the DLD  for $\lambda=1.1$ and $i=100$]{\includegraphics[width=0.3\linewidth]{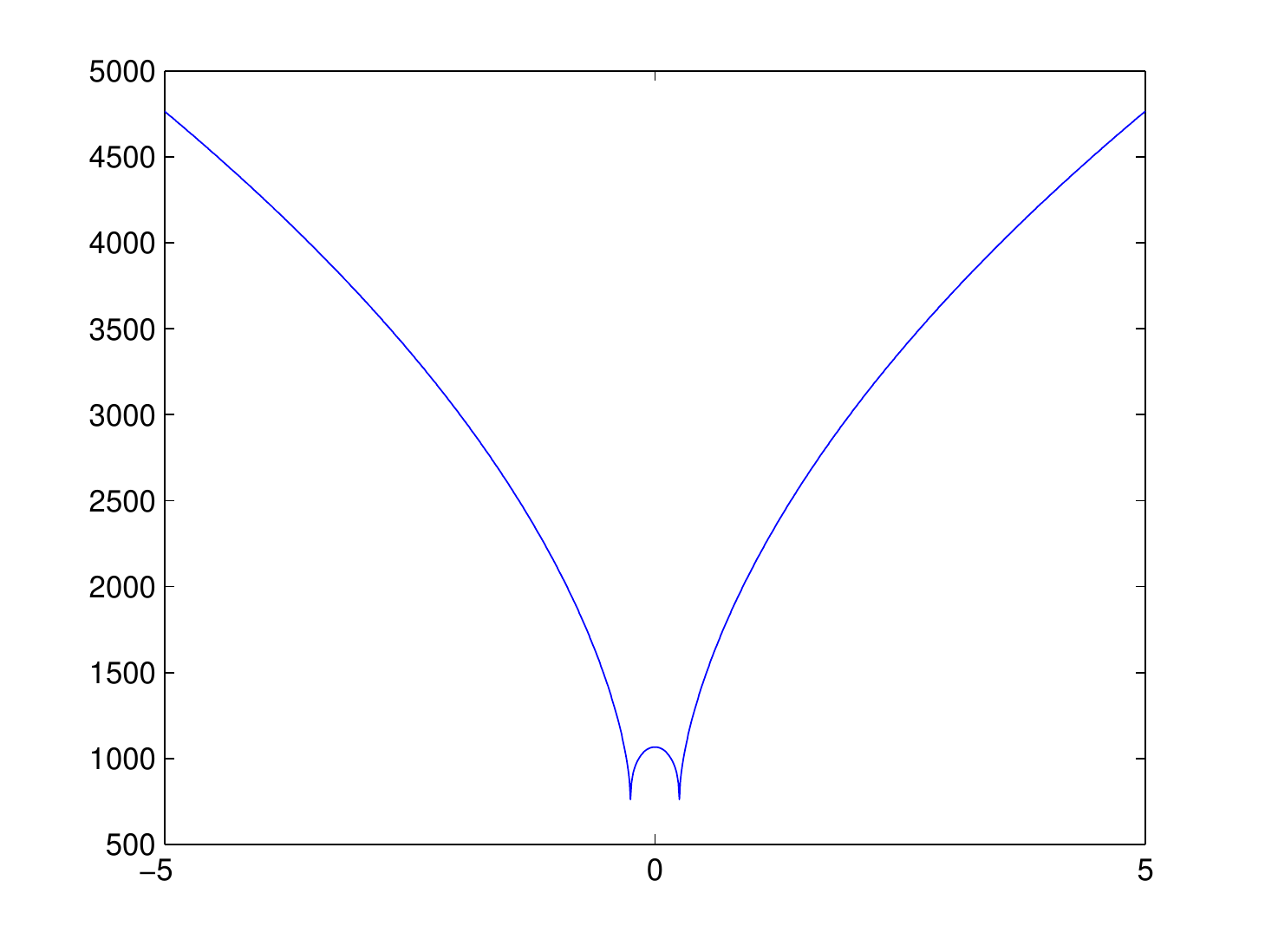}}\\
  \caption{Derivative of the DLD  along the line $y=0.25$ for different values of $\lambda$ and iterations $i$.}
  \label{derivative_of_DLD}
\end{figure*}

\subsection{Example 2: A Hyperbolic Saddle Point for Nonlinear, Area-Preserving Autonomous Maps}
\label{sec:examp2}

We will analyze this case using a theorem of \cite{moser56}. Moser's theorem applies to analytic, area
preserving maps in a neighborhood of a hyperbolic fixed point. We will discuss how the assumptions of analyticity and
area preservation can be removed  later on, but for now we proceed with these assumptions.

We consider an analytic, area-preserving map in a neighborhood of $x=y=0$ of the form:

\begin{equation}\label{nonlinear_equation}
\left \{ \begin{array}{ccccc}
   x_{n+1} & = & f(x_n,y_n) & = & \lambda  x_n + \cdots\\
   y_{n+1} & = & g(x_n,y_n) & = & \lambda^{-1}  y_n + \cdots\\
\end{array}\right .
\end{equation}

\noindent
where $\lambda>1$ and $''\cdots''$ represent nonlinear terms that obey the area-preserving constraint. Moser's Theorem
states that there exists a real analytic, area preserving change of variables of the following form:

\begin{equation}\label{change_variables}
\begin{array}{ccccc}
   x & = & x(\xi,\eta), \\
   y & = & y(\xi,\eta), \\
\end{array}
\end{equation}

\noindent
with inverse

\begin{equation}\label{change_variables_inverse}
\begin{array}{ccccc}
   \xi & = & \xi(x,y), \\
   \eta & = & \eta(x, y), \\
\end{array}
\end{equation}

\noindent
such that in these new coordinates
\eqref{nonlinear_equation} has the following {\em normal form}:

\begin{equation}\label{normal_form}
\left \{ \begin{array}{ccl}
   \xi_{n+1}  & = & U(\xi_n\eta_n)\xi_n\\
   \eta_{n+1} & = & U^{-1}(\xi_n\eta_n)\eta_n\\
\end{array}\right .
\end{equation}

\noindent
where $U(\xi\eta)$ is a power series in the product $\xi \eta$ of the form $U_0 + U_2\xi\eta + \cdots $, with $U_0 =
\lambda$, which converges  in a neighborhood of the hyperbolic point. Note that it follows from the form of
\eqref{normal_form} that $U(\cdot)$ is constant on orbits of \eqref{normal_form}, i.e. $U(\xi_{i+1}
\eta_{i+1}) = U(\xi_i \eta_i) = U, \forall i$.

The form of \eqref{normal_form} implies that the same computation described in Section \ref{sec:examp1} applies.
Therefore for $MD^{+}_p$ we have:

\begin{eqnarray}
  MD^{+}_p & = & \displaystyle{\sum ^{N-1}_{i=0}  |\xi_{i+1}-\xi_i|^p + |\eta_{i+1}-\eta_i|^p} \nonumber\\
	    &   & 						     \nonumber\\
            & = & \displaystyle{\sum ^{N-1}_{i=0}  |\xi_i|^p|U(\xi_i\eta_i)-1|^p + |\eta_i|^p|U^{-1}(\xi_i \eta_i)-1|^p
= \sum ^{N-1}_{i=0}  |\xi_i|^p|U-1|^p + |\eta_i|^p|U^{-1}-1|^p} \nonumber\\
            &   &						    \nonumber \\
            & = & \displaystyle{|\xi_0|^p|U-1|^p \left (1+|U|^p+...+|U|^{(N-1)p}\right ) + |\eta_0|^p|U^{-1}-1|^p \left
(1+|U^{-1}|^p+...+|U^{-1}|^{(N-1)p}\right )} \nonumber\\
            &   &						   \nonumber  \\
            & = & \displaystyle{|\xi_0|^p|U-1|^p\left | \frac{U^{Np}-1}{U^p-1} \right | +
|\eta_0|^p|U^{-1}-1|^p\left | \frac{1/U^{Np}-1}{1/U^p-1} \right |}. \nonumber
\end{eqnarray}
\normalsize

\noindent
$MD^{-}_p$ is computed analogously, and therefore $MD_p = MD^{+}_p + MD^{-}_p$
is given by:

$$ MD_p  = \displaystyle{(|\xi_0|^p + |\eta_0|^p) \left (|U-1|^p\left
| \frac{U^{Np}-1}{U^p-1} \right | +|U^{-1}-1|^p\left
| \frac{1/U^{Np}-1}{1/U^p-1} \right | \right )},$$
\normalsize
\noindent
In this expression $U$ is constant along trajectories, {\em i.e.},  $U(\xi_0\eta_0) = U(\xi_i\eta_i) = U, \forall
i$. But in general, different initial conditions $( \xi_0, \eta_0)$ do not belong to the same trajectory, thus $U$ 
depends on $( \xi_0, \eta_0)$.
More succinctly we express this as:
\begin{equation}
\begin{array}{ccl}
   MD_p & = & \displaystyle{(|\xi_0|^p + |\eta_0|^p)f(U(\xi_0,\eta_0),p,N)}\\
\end{array}
\label{eq:DLD_nonlin_aut}
\end{equation}

\noindent
This expression has the same form as \eqref{DLD_lin_aut}, except for the dependence of the function $f$ on 
$U(\xi_0,\eta_0)$. We note that  $U$ is analytical and thus it is a smooth function. Therefore Theorem 
\ref{thm:lin_aut_map} still applies because  the first derivative is infinite
due  to the  first factor  in  expression (\ref{eq:DLD_nonlin_aut}). We can conclude that the derivative of $MD_p$ 
transverse to the stable manifold is singular on the manifold and the derivative of $MD_p$ transverse to the unstable 
manifold is singular on the manifold. However, this is a statement that is true in the $\xi-\eta$ normal form 
coordinates. In practice we will compute the Lagrangian descriptor in the original $x-y$ coordinates and therefore we 
would like to conclude that the ``singular sets''  of the Lagrangian descriptor in the $x-y$ coordinates correspond to 
the stable and unstable manifolds of the hyperbolic fixed point. We will now show that this is the case. We will carry 
out the argument for the the stable manifold. The argument for the unstable manifold is completely analogous.

First, using \eqref{change_variables}, in the $x-y$ coordinates the stable manifold of the origin is given by the curve
$(x(0, \eta), y(0, \eta))$. Here $\eta$ is viewed as a parameter for this parametric representation of the stable
manifold in the original $x-y$ coordinates. A vector perpendicular to this curve at any point on the
curve is given by $\left( -\frac{dy}{d \eta} (0, \eta), \frac{dx}{d \eta} (0, \eta) \right)$. Now we compute the rate
of change of $MD_p=MD_p (x, y)$ in this direction and consider its behavior on the stable manifold of the origin.This
is given by the directional derivative of $MD_p (x, y)$  in this direction evaluated on the stable manifold:

\begin{equation}
\left( \frac{\partial MD_p}{\partial x}(x(0, \eta), y(0, \eta)), \frac{\partial MD_p}{\partial y} (x(0, \eta), y(0,
\eta))\right) \cdot
\left( -\frac{dy}{d \eta} (0, \eta), \frac{dx}{d \eta} (0, \eta) \right),
\label{eq:direc_deriv}
\end{equation}

\noindent
where the derivatives are evaluated on $(x(0, \eta), y(0, \eta))$, but we will omit this explicitly for the sake of a
less cumbersome notation. Next we will use the chain rule to express partial derivatives with respect to $x$ and $y$ in
terms of $\xi$ and $\eta$ as follows:

\begin{eqnarray}
\frac{\partial MD_p}{\partial x} & = & \frac{\partial MD_p}{ \partial \xi}  \frac{\partial \xi}{\partial x} +
\frac{\partial MD_p}{ \partial \eta}  \frac{\partial \eta}{\partial x}, \nonumber \\
\frac{\partial MD_p}{\partial y} & = & \frac{\partial MD_p}{ \partial \xi}  \frac{\partial \xi}{\partial y} +
\frac{\partial MD_p}{ \partial \eta}  \frac{\partial \eta}{\partial y}.
\label{eq:cr}
\end{eqnarray}

\noindent
Substituting \eqref{eq:cr} into \eqref{eq:direc_deriv} gives:

\begin{eqnarray}
-\left( \frac{\partial MD_p}{ \partial \xi}  \frac{\partial \xi}{\partial x} + \frac{\partial MD_p}{ \partial \eta}
\frac{\partial \eta}{\partial x}\right) \frac{dy}{d \eta}  +
  \left( \frac{\partial MD_p}{ \partial \xi}  \frac{\partial \xi}{\partial y} + \frac{\partial MD_p}{ \partial \eta}
\frac{\partial \eta}{\partial y}\right) \frac{dx}{d \eta}.
\end{eqnarray}

Now it follows from the argument given in Theorem \ref{thm:lin_aut_map} that
$\frac{\partial MD_p}{ \partial \xi} $ is not differentiable on the stable manifold ($\xi =0$ for $p<1$).  Hence
\eqref{eq:DLD_nonlin_aut} is  not differentiable in a direction transverse to the stable manifold at a point on the
stable manifold in the $x-y$ coordinates.

\subsection{Example 3: A Hyperbolic Saddle Point for Linear, Area-Preserving Nonautonomous  Maps}
\label{sec:examp3}

In this section we will consider the nonautonomous analog of example 1 in Section \ref{sec:examp1}. Namely, we will
consider a linear, area preserving nonautonomous map having a hyperbolic trajectory at the origin. The map that we
consider has the following form:

$$\left \{ \begin{array}{ccc}
   x_{n+1} & = & \lambda_n x_n\\
   y_{n+1} & = & \frac{1}{\lambda_n} y_n \\
\end{array}\right .$$

\noindent
where $\lambda_n >1, \, \forall n$.   Note that $x =y =0$ is a hyperbolic trajectory with stable manifold given by
$x=0$ and unstable manifold given by $y=0$ {\em for all $n$}.

We will only compute $MD^+_p$ since the computation of $MD^-_p$ is analogous.   Hence, for $MD^+_p$  we have:

\begin{eqnarray}
   MD^{+}_p & = & \displaystyle{\sum ^{N-1}_{i=0}  |x_{i+1}-x_i|^p + |y_{i+1}-y_i|^p =
\sum^{N-1}_{i=0} |x_i|^p|\lambda_i-1|^p + |y_i|^p|1/\lambda_i-1|^p} \nonumber \\
            &   &						    \nonumber \\
            & = & \displaystyle{|x_0|^p\left (|\lambda_0-1|^p+|\lambda_0|^p|\lambda_1-1|^p +...+ |\lambda_0 \cdots
\lambda_{N-2}|^p|\lambda_{N-1}-1|^p\right ) + }\nonumber\\
            &   &                     \nonumber  \\
            &   & \displaystyle{|y_0|^p\left (|1/\lambda_0-1|^p+|1/\lambda_0|^p|1/\lambda_1-1|^p +...+ |1/\lambda_0
\cdots 1/\lambda_{N-2}|^p|1/\lambda_{N-1}-1|^p\right )} \nonumber\\
            &   &   				    \nonumber \\
            & = & \displaystyle{|x_0|^p \left ( |\lambda_0-1|^p + \sum^{N-1}_{i=1} \left
(\prod^{i-1}_{j=0}|\lambda_j|^p \right )|\lambda_i-1|^p\right ) +}  \nonumber \\
 & & \displaystyle{|y_0|^p \left ( |1/\lambda_0-1|^p + \sum^{N-1}_{i=1}  \left (\prod^{i-1}_{j=0}|1/\lambda_j|^p \right )|1/\lambda_i-1|^p\right ) } \nonumber
\end{eqnarray}
\normalsize
\noindent
A similar calculation gives:

\begin{eqnarray}
 MD^{-}_p & = & \displaystyle{|x_0|^p \left ( |1-1/\lambda_{-1}|^p + \sum^{-N}_{i=-2} \left
(\prod^{i+1}_{j=-1}|1/\lambda_j|^p \right )|1-1/\lambda_i|^p\right )} +\nonumber\\
& & \displaystyle{|y_0|^p \left ( |1-\lambda_{-1}|^p +
\sum^{-N}_{i=-2} \left (\prod^{i+1}_{j=-1}|\lambda_j|^p \right )|1-\lambda_i|^p\right )} .\nonumber
  \end{eqnarray}
\normalsize

\noindent
Combining these two expressions gives:

\begin{equation}
\begin{array}{ccl}
MD_p & = & \displaystyle{|x_0|^p f(\Lambda,p,N) + |y_0|^p g(\Lambda^{*},p,N)}\\
\end{array}
\label{DLD_lin_nonaut}
\end{equation}

\noindent
where
$$\Lambda = (\lambda_0,\lambda_1,...,\lambda_{N-1},1/\lambda_{-1},1/\lambda_{-2},...,1/\lambda_{-N})$$
and
$$\Lambda^{*} = (1/\lambda_0,1/\lambda_1,...,1/\lambda_{N-1},\lambda_{-1},\lambda_{-2},...,\lambda_{-N}).$$

Now \eqref{DLD_lin_nonaut} has the same functional form as \eqref{DLD_lin_aut}. So for $p<1$ the same argument as given
in Theorem \ref{thm:lin_aut_map} holds. Therefore, along a line transverse to the  stable manifold (i.e. $x=0$) $MD_p$
is not differentiable at the point on this line that intersects the stable manifold.  The analogous statement holds for
the unstable manifold.

\subsection{Example 4: A Hyperbolic Saddle Point for a Nonlinear, Area Preserving Nonautonomous Map}
\label{sec:examp4}

We now consider a two dimensional nonlinear area-preserving nonautonomous map having the following form:

\begin{eqnarray}
   x_{n+1} & = & \lambda_n x_n + f_n(x_n,y_n), \nonumber \\
   y_{n+1} & = & \lambda_n^{-1} y_n + g_n(x_n,y_n),  \quad (x_n, y_n) \in \mathbb{R}^2, \forall n,
 \label{nonlinear_nonautonomous}
\end{eqnarray}

\noindent
where $\lambda_n >1, \, \forall n$ with $f_n (0, 0) = g_n (0, 0)=0, \, \forall n$. We assume that $f_n (\cdot, \cdot)$
and $g_n (\cdot, \cdot)$ are  real valued nonlinear functions (i.e. of order quadratic or higher), they are at least
$C^1$, and they satisfy the constraints that the nonlinear map defined by \eqref{nonlinear_nonautonomous} is area
preserving.

Since the origin is a hyperbolic trajectory it follows that it has (one dimensional) stable and unstable manifolds
(\cite{irwin,deblasi,KH}). We will apply the method of discrete
Lagrangian descriptors to \eqref{nonlinear_nonautonomous} and show that the stable and unstable manifolds of the origin
correspond to the ``singular features'' of $MD_p$ ($p<1$), in the sense described in Theorem \ref{thm:lin_aut_map}.
Our method of proof will be similar in spirit to how we showed  the result for nonlinear autonomous maps by using
Moser's theorem. Unfortunately, there is
no analog of Moser's theorem for nonlinear, nonautonomous area preserving two dimensional maps. Nevertheless, we will
still use a ``change of variables'', or ``conjugation'' result that is a nonautonomous  map version of the
Hartman-Grobman theorem due to \cite{bv06}.

The  classical Hartman-Grobman (\cite{hart60a,hart60b,hart63,grob59,grob62}) theorem applies to autonomous maps in a
neighborhood of a hyperbolic fixed point. The result states that there exists a homeomorphism, defined in a
neighborhood of the fixed point, which conjugates the map to its linear part. Stated another way, the homeomorphism
provides a new set of coordinates where the map is given by its linear part in the new coordinates. There are two issues
that we must immediately face in order for this approach to work as it did for the linear and nonlinear autonomous maps.
One is the generalization of the Hartman-Grobman theorem to  the setting on nonautonomous maps (this is dealt with  in
\cite{bv06}) and the other is the smoothness of the conjugation (``change of coordinates'') since a derivative is
required in the application of the chain rule (see \eqref{eq:cr}).

In general, the conjugacy provided by the Hartman-Grobman theorem is not differentiable (see \cite{meyer86} for
examples). However, there has been much work in determining conditions under which the conjugacy is at least $C^1$,
see, e.g., \cite{svs90,ghr03}. Moreover, Hartman has proven (\cite{hart60b}) that in two dimensions, a $C^2$
diffeomorphism having a  hyperbolic saddle can be linearized with a $C^1$ conjugacy (see also \cite{s86}). We also point
out that differentiability is a property defined pointwise, and the nondifferentiability of the conjugacy typically
fails to hold at the fixed point (see the examples in \cite{meyer86}) and we are not interested in  differentiability at
the fixed point, but at points along the stable and unstable manifolds of the fixed point. The conjugacy is
differentiable at these points, as is described in the lecture notes of Rauch entitled ``Conjugacy Outline'' availiable
at
\url{http://www.math.lsa.umich.edu/~rauch/courses.html}. This result also follows from the rectification theorem for
ordinary differential equations  (\cite{arnold73}) which says that,  away from points where the vector field vanishes,
the vector field is conjugate to ``rectilinear flow'', and this conjugacy is as smooth as the vector field. Note that
this result is valid for both autonomous and nonautonomous vector fields.

So setting aside the smoothness issues, we will give a brief discussion of the set-up of \cite{bv06} for the
nonautonomous Hartman-Grobman theorem.  They consider
that the  phase space is given by a Banach space, denoted $X$ (for us $X$ is $\mathbb{R}^2$). The dynamics is described
by a sequence of maps on $X$:

\begin{equation}
F_n (v) = A_n v + f_n (v),  \quad v \in X, \, n \in \mathbb{Z}.
\label{eq:genNAmap}
\end{equation}

\noindent
Precise assumptions on $A_n$ and $f_n (v)$ are given in \cite{bv06}).  In particular $A_n$ is a hyperbolic operator,
which for us is:

\begin{equation}
A_n = \left(
\begin{array}{cc}
\lambda_n & 0 \\
0 & \lambda_n^{-1}
\end{array}
\right)
\end{equation}

\noindent
and  where $f_n (v)$ is ``small'', in some sense, e.g. $f_n (0)=0$ with $f_n (v)$ satisfying a Lipschitz condition. Our
$f_n (v)$ will be at least $C^1$ and satisfy  the condition for the map  \eqref{nonlinear_nonautonomous} to be area
preserving.

For each $n \in \mathbb{Z}$  construct a homeomorphism, $h_n (\cdot)$ that conjugates
\eqref{eq:genNAmap} to its linear part, i.e.,

\begin{equation}
A_n \circ h_n = h_{n+1} \circ F_n,
\end{equation}

\noindent
or, expressing this in a diagram for the full dynamics (following \cite{bv06}) we have:

\begin{equation}
\begin{array}{clclclclc}
 & & F_{n-1} & & F_n & & F_{n+1} & & \\
 \longrightarrow & X & \longrightarrow & X & \longrightarrow  & X & \longrightarrow  & X & \longrightarrow\\
  & \downarrow h_{n-1} &  &\downarrow h_{n}& &\downarrow h_{n+1}  &  & \downarrow h_{n+2} &\\
  & &   A_{n-1} & &  A_n & & A_{n+1} & &  \\
  \longrightarrow & X & \longrightarrow & X & \longrightarrow  & X & \longrightarrow  & X & \longrightarrow
\end{array}
\end{equation}

In Section \ref{sec:examp3} we proved that the discrete Lagrangian descriptor for the  linear, area preserving
nonautonomous map is singular along the stable and unstable manifolds of the hyperbolic trajectory at the origin, i.e.
$x=0$ and $y=0$, respectively. Note that the discrete Lagrangian descriptor is only a function of the initial
condition, $(x_0, y_0)$. Hence we can use the change of coordinates $h_0 (\cdot)$ and the argument given in Section
\ref{sec:examp2} to conclude that the discrete Lagrangian descriptor for the nonlinear nonautonomous area preserving
map \eqref{nonlinear_nonautonomous}  is singular along the stable and unstable manifolds.

\section{Application to the Chaotic Saddle of the  H\'enon Map}
\label{sec:henon}

We  now  illustrate the method of discrete Lagrangian descriptors for autonomous, area preserving nonlinear maps by
applying it to the H\'enon map (\cite{henon76}):

\begin{equation}
H(x,y) = (A+By-x^2,x).
\label{eq:henonmap}
\end{equation}

\noindent
The map is area preserving for $|B|=1$ and is orientation-preserving if $B<0$. Moreover, it
follows from work in \cite{dn79} that for values of $A$ larger than

\begin{equation}
A_{2} = (5 + 2\sqrt{5})(1 + |B|)^2/4,
\label{eq:ACM}
\end{equation}

\noindent
the H\'enon map has a hyperbolic invariant Cantor set which is topologically conjugate to a  Bernoulli shift on two 
symbols, i.e. it has a {\em chaotic saddle}. We will use the method of discrete Lagrangian descriptors to visualize this 
chaotic saddle.

We  consider $B=-1$, which after substituting this value into \eqref{eq:ACM}, gives $A_{2} = 5 + 2\sqrt{5} \approx 
9.47$, and therefore we choose $A=9.5$, which satisfies the chaos condition. With these choices of parameters we have 
$H(x,y) = (9.5-y-x^2,x)$.  Applying the method of discrete Lagrangian descriptors to this map gives the structures shown 
in  Figure \ref{dibujo_chaotic_saddle_Henon}, where the chaotic saddle is the set that appears as dark blue. This 
method, in contrast to other techniques for computing chaotic saddles (see for instance \cite{yorke}), has the advantage 
that it simultaneously provides insight into the manifold structure associated with the chaotic saddle.

  \begin{figure}[ht]
    \centering
    \includegraphics[scale=0.65]{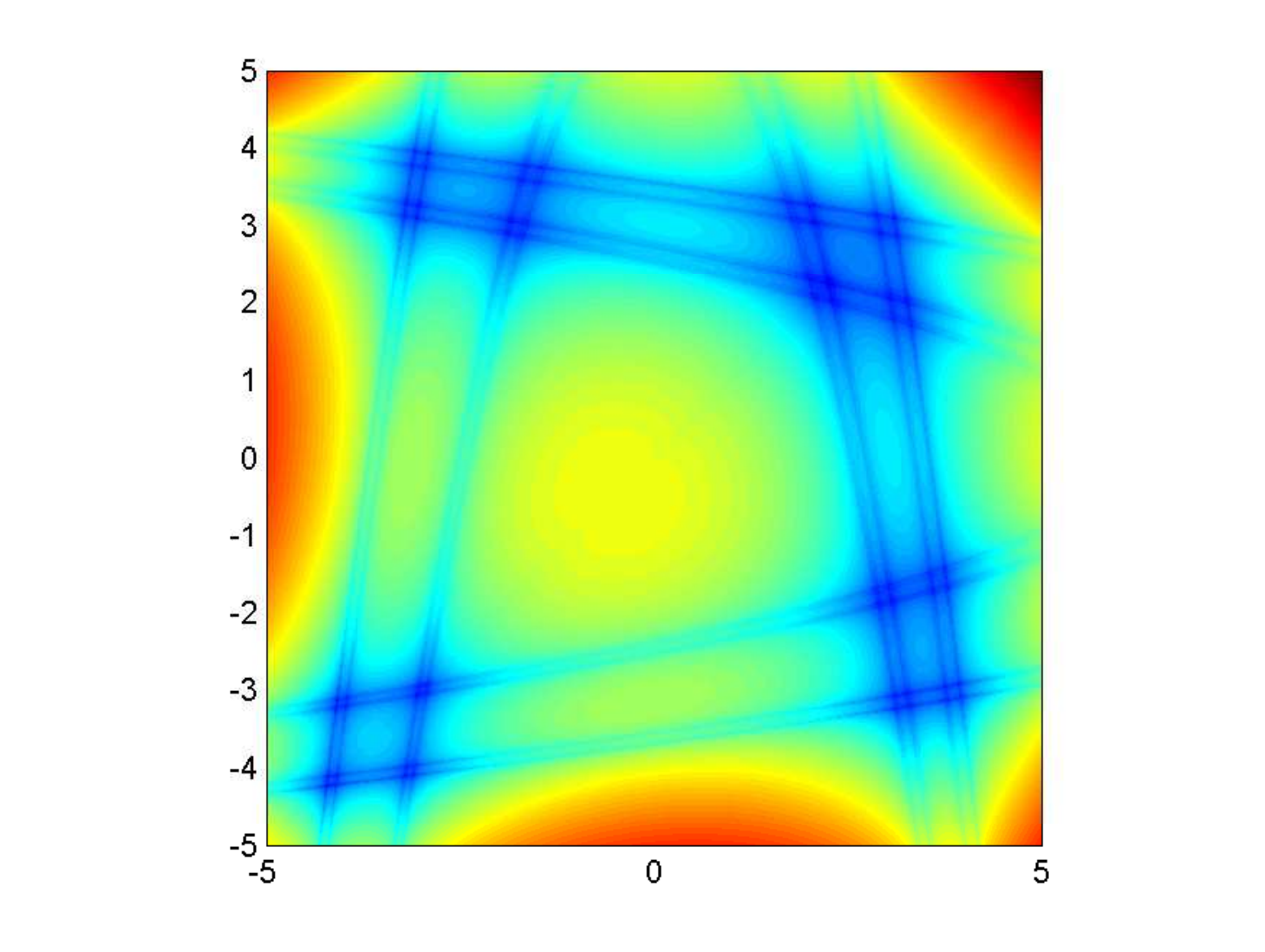}
    \caption{Computation of the chaotic saddle of the H\'enon map for $A =9.5, \, B = -1$, after $N=5$ iterations and 
$p=0.05$.}
    \label{dibujo_chaotic_saddle_Henon}
  \end{figure}

\section{Application to the Chaotic Saddle of a Nonautonomous  H\'enon Map}
\label{sec:NAhenon}

We  now  illustrate the method of discrete Lagrangian descriptors for nonautonomous, area preserving maps by
applying it to a nonautonomous version of the  H\'enon map. In particular, in \eqref{eq:henonmap} we take;

\begin{equation}
B= -1, \quad A = 9.5 +\epsilon \, \cos (n).
\end{equation}

\noindent
For $\epsilon$ `small'', this is a nonautonomous perturbation of the situation considered in Section \ref{sec:henon}, 
so that  we would expect to have a structure similar to that shown in Figure \ref{dibujo_chaotic_saddle_Henon}, but  
slightly varying with $n$, i.e. a nonautonomous chaotic saddle (see \cite{wiggins99}).

The discrete Lagrangian descriptor method provides us with a numerical tool to explore this question.
Figure \ref{fig:NAhenon}  illustrates the phase space structure at different times for the nonautonomous H\'enon map.
 Clearly the output is similar to that shown in Fig. \ref{dibujo_chaotic_saddle_Henon}, but varying with respect to $n$.

\begin{figure*}[htbp!]
\centering
\subfigure[ $n=-3$]
{\label{fig:good_array}\includegraphics[width=0.45\linewidth]{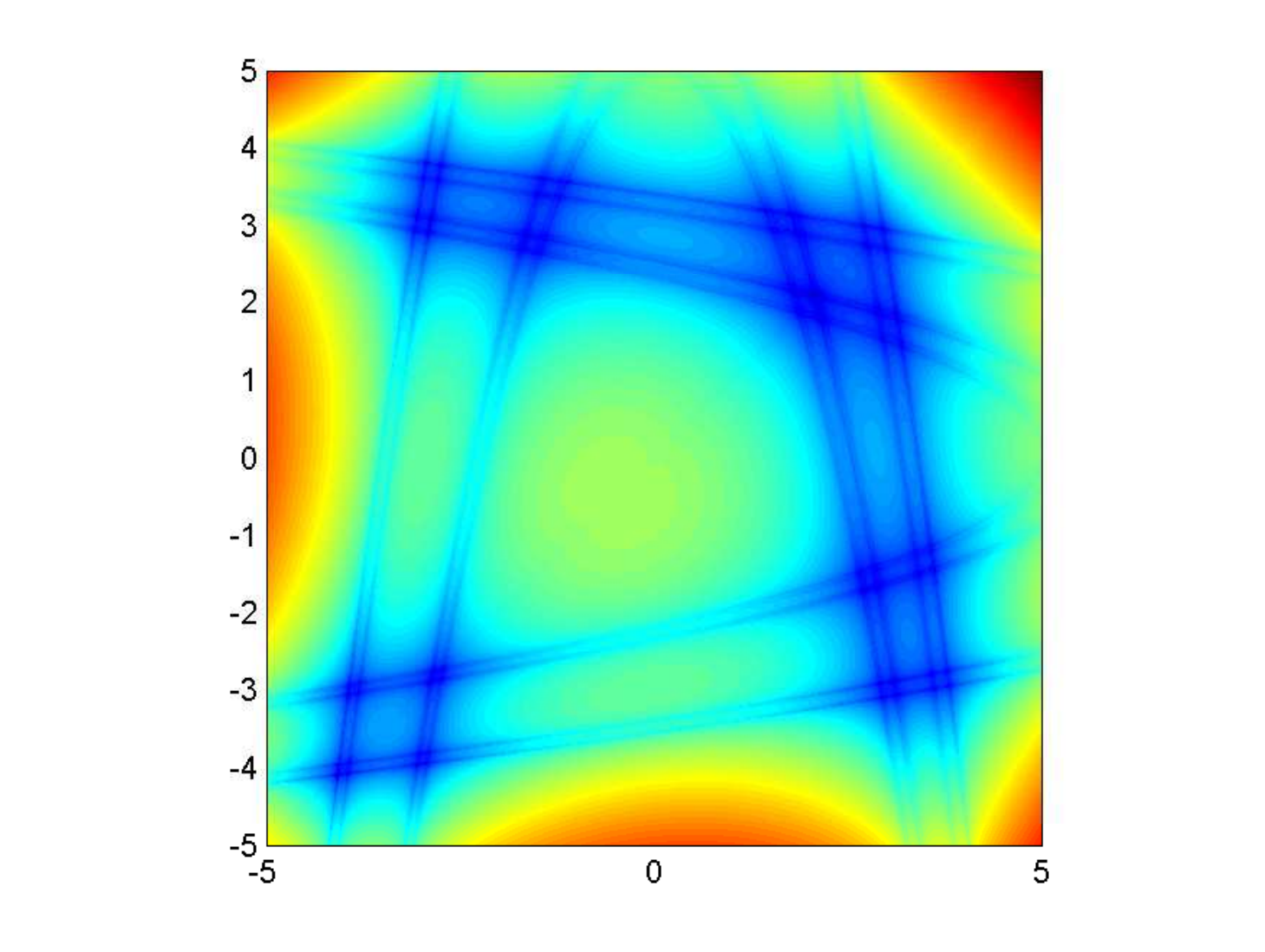}} 
\subfigure[ $n=-1$]
{\label{fig:bad2_array}\includegraphics[width=0.45\linewidth]{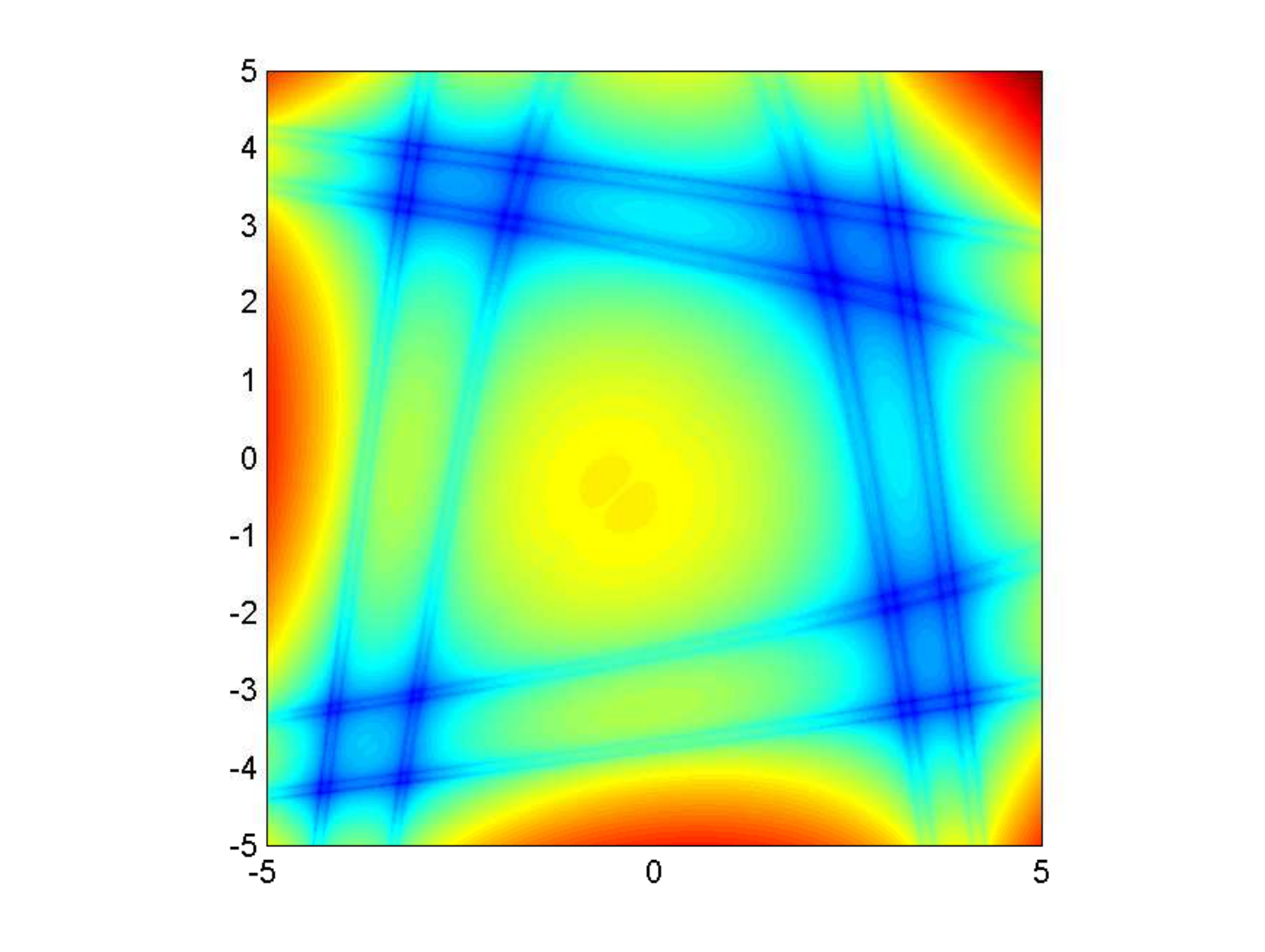}} \\
\subfigure[ $n=1$]
{\label{fig:good_array}\includegraphics[width=0.45\linewidth]{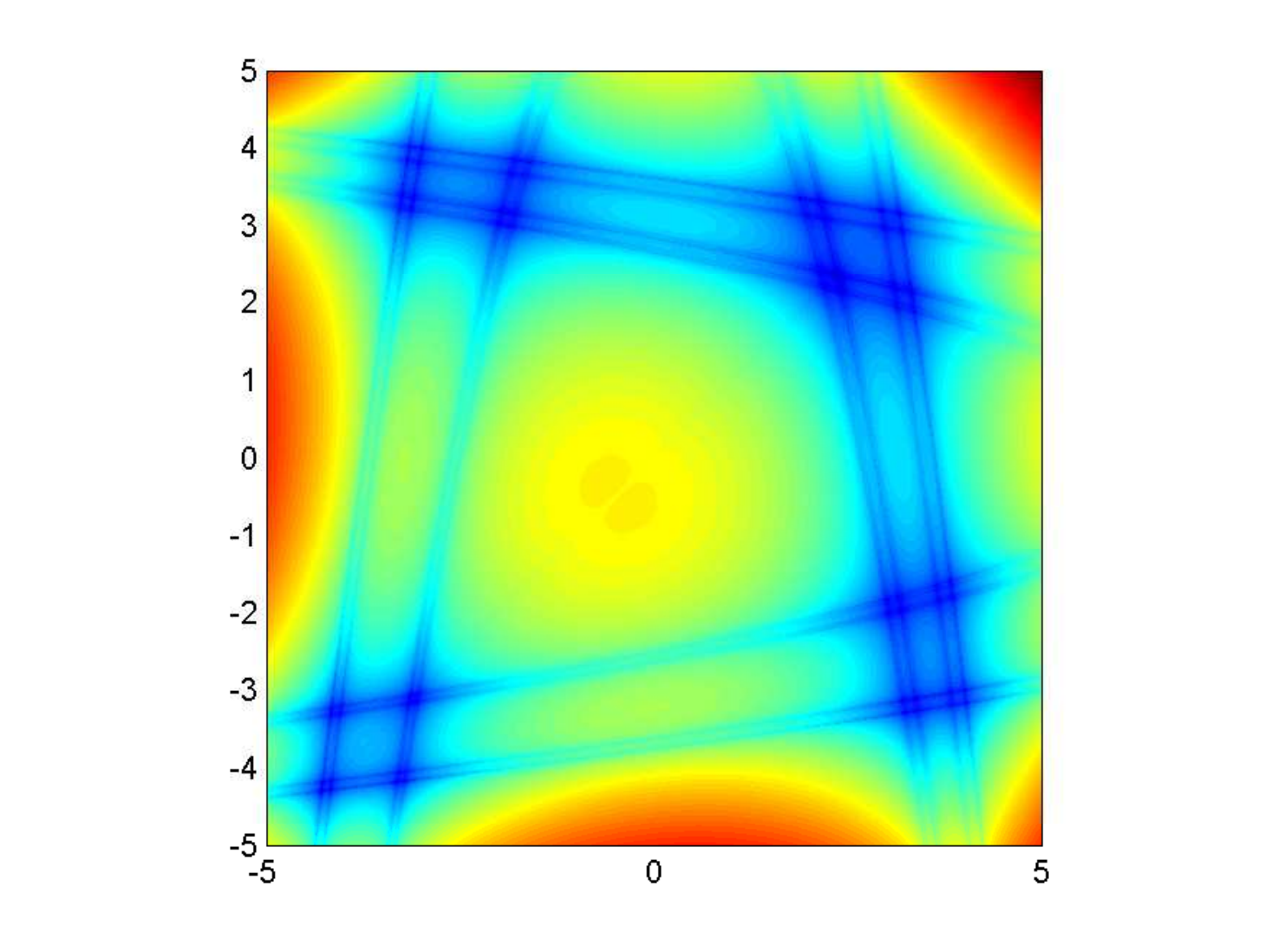}}
\subfigure[ $n=3$]
{\label{fig:bad2_array}\includegraphics[width=0.45\linewidth]{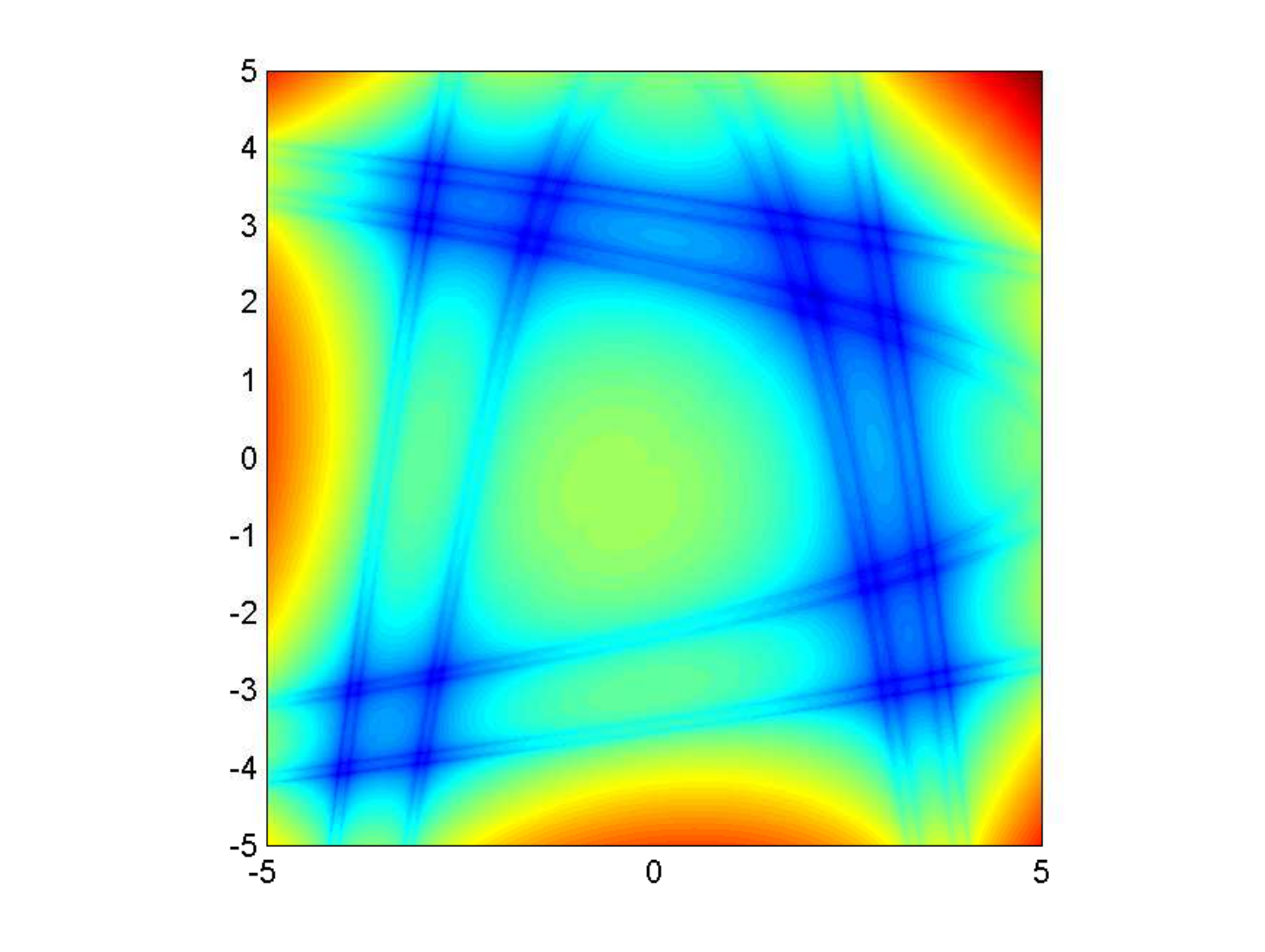}}\qquad\hfill\caption{Computation of the 
chaotic saddle of the nonautonomous H\'enon map for $A =9.5 +\epsilon \, \cos (n), \, B = -1$, after  $N=5$ iterations 
and $p=0.05$. The output is shown for four different times.}
\label{fig:NAhenon}
\end{figure*}

\section{Summary and Conclusions}
\label{sec:summ}

In this paper we have generalized the notion of Lagrangian descriptors to  autonomous and nonautonomous maps. We have 
restricted our discussion to two dimensional, area preserving maps, but with additional work it should be possible to 
remove these restrictions.

In the discrete time setting explicit expressions for the Lagrangian descriptors were derived, and for the $\ell^p$ 
norm, $p <1$, we  proved a theorem that gave rigorous meaning to the statement that ``singular sets'' of the Lagrangian 
descriptors  correspond to the stable and unstable manifolds of hyperbolic invariant sets.

\section*{\bf Acknowledgments.}  The research of AMM is supported by the MINECO under grant MTM2011-26696. The research 
of SW is supported by  ONR Grant No.~N00014-01-1-0769.  We acknowledge support from MINECO: ICMAT Severo Ochoa project 
SEV-2011-0087.

\end{document}